\documentclass{article}

\usepackage{subfigure}
\usepackage{amsfonts}
\usepackage{amssymb}
\usepackage{amsmath,mathrsfs}
\usepackage{amsthm}
\usepackage{latexsym}
\usepackage{layout}
\usepackage{graphics}
\usepackage{enumerate}
\usepackage{epsfig}

\usepackage[english]{babel}

\theoremstyle{plain}
\begingroup
\newtheorem{theorem}{Theorem}[section]
\newtheorem{lemma}[theorem]{Lemma}
\newtheorem{proposition}[theorem]{Proposition}
\newtheorem{corollary}[theorem]{Corollary}
\endgroup

\theoremstyle{definition}
\begingroup
\newtheorem{definition}[theorem]{Definition}
\newtheorem{remark}[theorem]{Remark}

\endgroup

\theoremstyle{remark}
\begingroup

\endgroup

\numberwithin{equation}{section}

\newcommand{\ran}{{\rm ran}}
\newcommand{\R}{{\mathbb R}}

\newcommand{\Om}{\Omega}

\newcommand{\J}{\mathcal{J}}
\newcommand{\F}{\mathcal{F}}

\newcommand{\N}{\mathbb N}

\newcommand{\EE}{{\mathcal E}}

\newcommand{\be}{\begin{equation}}
\newcommand{\ee}{\end{equation}}
\newcommand{\bes}{\begin{eqnarray}}
\newcommand{\ees}{\end{eqnarray}}

\title{Linearly Constrained Nonsmooth and Nonconvex Minimization}

\author{Marco Artina\thanks{Faculty of Mathematics, Technische Universit\"at M\"unchen, Boltzmannstrasse 3, 85748, Garching, Germany Email: {\tt marco.artina@ma.tum.de}} \and Massimo Fornasier\thanks{Faculty of Mathematics, Technische Universit\"at M\"unchen, Boltzmannstrasse 3, 85748, Garching, Germany Email: {\tt massimo.fornasier@ma.tum.de}} \and Francesco Solombrino\thanks{Faculty of Mathematics, Technische Universit\"at M\"unchen, Boltzmannstrasse 3, 85748, Garching, Germany Email: {\tt francesco.solombrino@ma.tum.de}}}

\begin{document}

\maketitle

\begin{abstract}
Motivated by variational models in continuum mechanics, we introduce a novel algorithm to perform nonsmooth and nonconvex minimizations with linear constraints in Euclidean spaces. We show how this algorithm is actually a natural generalization of the well-known non-stationary augmented Lagrangian method for convex optimization. The relevant features of this approach are its applicability to a large variety of nonsmooth and nonconvex objective functions, its guaranteed convergence to critical points of the objective energy independently of the choice of the initial value, and its simplicity of implementation. In fact, the algorithm results in a nested double loop iteration. In the inner loop an augmented Lagrangian algorithm performs an adaptive finite number of iterations on a fixed quadratic and strictly convex perturbation of the objective energy, depending on a parameter which is adapted by the external loop. To show the versatility of this new algorithm, we exemplify how it can be used for computing 
critical points in inverse free-discontinuity variational models, such as the Mumford-Shah functional, and, by doing so, we also derive and analyze new iterative thresholding algorithms. 
\end{abstract}

\\


\pagestyle{myheadings}
\thispagestyle{plain}
\markboth{M. ARTINA, M. FORNASIER, AND F. SOLOMBRINO}{LINEARLY CONSTRAINED NONSMOOTH AND NONCONVEX MINIMIZATION}

\section{Introduction}

Minimizers of integrals in calculus of variations may possess singularities, which typically arise as the result of the nonsmoothness or nonconvexity of the energy. For certain problems in continuum mechanics, such singularities represent physically interesting instabilities, like relevant features of solid phase transformations and certain modes of fracture (see e.g., \cite{frma98,mi02,yefrscXX}).
In this context, {\it local} minimizers of {\it nonconvex} energies play a pivotal role, as often evolution of physical phenomena proceeds along such energy critical points. Furthermore, usually the given problems have additional conditions, for instance boundary conditions, to be taken into account, which result in constraints, often of linear type, to be satisfied by the critical points. Therefore the appropriate solution of constrained genuinely {\it nonconvex} optimization problems is of the utmost interest as well as the accurate numerical treatment of the singularities which are expected to characterize the critical points.

In the literature one can find efficient algorithmic solutions for linearly constrained {\it convex} and {\it nonsmooth} minimization, e.g.,  augmented Lagrangian methods \cite{ber,poltret,Bre1,FrSch,ItoKun}, and for linearly constrained {\it nonconvex} minimization, such as sequentially quadratic programming (SQP) or (semi-smooth) Newton methods \cite{nowr06}. Unfortunately, in the latter cases only {\it smooth objective energies}, usually at least $C^2$ functionals can be addressed by algorithms, which are then guaranteed to converge {\it only locally} around the expected critical point. A more general setting is the one  considered in \cite{ABS}, where a remarkable analysis of the convergence properties of descent methods for nonconvex optimization, also with constraints, has been carried out. A key role in the paper is played by a special condition, the so-called Kurdyka-{\L}ojasiewicz inequality (see for instance \cite{Bdl}) allowing for a general convergence result in a nonsmooth nonconvex 
setting, but 
again under 
the 
assumption of a {\it good initial guess}. To remove the latter very restrictive assumption, a certain smoothness is needed, namely, {\it global} $C^{1,1}$ regularity. Let us also stress that, although quite a mild condition from the point of view of the applications, the Kurdyka-{\L}ojasiewicz inequality could {\it not} be verified even in the case of convex functions, as shown again in \cite{Bdl}.

The above mentioned limitations of the currently available literature lead us to the motivation of this paper. Its first goal is to propose a {\it very general} and {\it simple} iterative algorithm to solve nonsmooth and nonconvex optimization problems with linear constraints. For nonsmoothness we mean that we require our objective function to be in general only a {\it locally} Lipschitz function, contrary to the much more restrictive $C^2$, or $C^{1,1}$ regularity requested by most of the above mentioned known methods for providing convergence guarantees, as in \cite{ABS}. 
Moreover, as one of the most relevant features of our iteration, we will show its {\it unconditionally guaranteed} convergence. By this we mean that the initial state does not need to be in a small neighborhood of a critical point. Our algorithm may in fact be viewed as an appropriate combination of the above mentioned techniques, resulting in a nested double loop iteration, where in the inner loop an augmented Lagrangian algorithm, or Bregman iteration, performs an adaptive finite number of iterations on a fixed {\it local} quadratic perturbation of the objective energy around the previous iteration, while the external loop performs an adaptation of the quadratic perturbation, similarly to SQP. Our analysis of convergence is confined to the setting of finite dimensional Euclidean spaces. Nevertheless, most of it could be done in the more general framework of (possibly 
infinite dimensional) Hilbert spaces since the only point where finite dimensionality is actually needed, is to recover strong compactness in the proof of Theorem \ref{mainthm}. 
\\
In the second part of this paper, we show the versatility of this algorithm by discussing some relevant applications. Our attention goes in particular to nonsmooth and nonconvex functionals of the type
\begin{equation}\label{examples} 
\J(v) = \| T v - g \|^2 + \gamma \sum_{k=1}^m  U_k(v_k),
\end{equation} 
subject to a linear constraint $A v = f$. Here $\gamma >0$ is a positive regularization parameter, $g$ is a datum, $T$ is a linear functional, and $U_k:\mathbb R \to \mathbb R_+$, for $k=1,\dots,m$, are scalar nonconvex maps acting on the components $(v_k)_{k=1}^m$ of the vector $v$ with respect to a fixed basis in a Euclidean space of dimension $m$. Among the models that after discretization fit the general optimization problem \eqref{examples} we present the Mumford-Shah functional in image processing, and the energy functionals driving the evolution of elastic bodies in well-established models of cohesive \cite{ca08,cato11,daza07} and brittle fracture \cite{bo07,dato02,dato02-1,frma98}.  This list is far from being complete, as we actually expect that the algorithm we study in this work can have significant further numerical applications also in other problems involving nonsmooth and nonconvex energies with additional linear (boundary) conditions, like elasto-plastic evolutions \cite{daro10,
mi02}, and atomic structure computations \cite{yefrscXX}. 

With the scope of clarifying in detail the applicability of our algorithm,  we focus on problems of the type \eqref{specfunc} for $U_k = W_r^p$, where $W_r^p(t) = \min \{|t|^p, r^p\}$, for $r>0$, $p\geq 1$, and any $t \in \mathbb R$. The choice of analyzing the case of this truncated polynomial potential is motivated by its particular relevance after appropriate discretization for several applications, e.g., in image processing, quasi-static evolutions of brittle fractures, compressed sensing, etc. We shall furnish more details on such a modeling in Section \ref{applications}. Furthermore, as a guideline to users,  the analysis of the related minimization problem gives us the possibility of discussing in detail  the role of the coercivity of the objective functionals as well as of the main conditions (A1) and (A2) appearing in our abstract analysis of convergence as presented in Section \ref{algorithm}. In particular, we show in this concrete situation, which we consider as a relevant template for several 
other cases as mentioned above, how such assumptions can in fact be properly fulfilled.
In the context of a truncated polynomial potential, our analysis requires a smooth perturbation technique which is reminiscent of previous methods of continuation-based deterministic relaxation, such as the {\it graduated nonconvexity} (GNC) pioneered by Blake and Zisserman \cite{BZ} in the context of the Mumford-Shah model, see also recent developments in \cite{niidmo98,nikolova99, niko2, niko3} and references therein. Being of some conceptual relevance for the scope of this paper, we mention how this latter technique works. For a suitable parameter $\varepsilon \in [0,1]$, one considers a continuous family of smoother objectives $\mathcal J^\varepsilon$ such that $\lim_{\varepsilon \to 1} \mathcal J^\varepsilon =  \mathcal J$ (at least pointwise), where $\mathcal J$ is the nonconvex energy to be minimized. Then one addresses the 
global minimization 
of $\mathcal J$ by iterated {\it local} minimizations along  $\mathcal J^\varepsilon$ when $\varepsilon$ is increasing from $0$ to $1$ with a strictly convex initial $\mathcal J^0$. More formally, we consider an increasing sequence $(\varepsilon_n)_{n \in \mathbb N}$, with $\varepsilon_0 =0$ and $\lim_n \varepsilon_n =1$ and the iterative algorithm
\begin{equation}\label{detcont}
v^{n+1} = \arg \min_{v \in \mathcal N_{\varepsilon_n} (v^n)} \mathcal J^{\varepsilon_n} (v),
\end{equation}
where $\mathcal N_{\varepsilon_n} (v^n)$ is a suitable neighborhood of the previous iteration $v^n$ of size possibly depending on $\varepsilon_n$. While such semi-heuristic algorithms perform very well in practice, usually they do not provide eventually any guarantee for global convergence and their applicability highly depends on the appropriate design of the approximating family $\{ \mathcal J^\varepsilon: \varepsilon \in [0,1]\}$, depending on the particular application and form of $\mathcal J$. Our algorithm has instead {\it more general applicability} and {\it stronger convergence guarantees}, providing as a byproduct also some rigorous justification to those semi-heuristic methods. \\
Another interesting feature of the application of the proposed algorithm to linearly constrained nonsmooth and nonconvex minimization involving truncated polynomial energy terms, is that the {\it inner loop} can be realized by means of an {\it iterative thresholding algorithm}. This technique has been as first proposed in \cite{fowa10} to solve inverse free-discontinuity problems in one dimension, where no approximate smoothing of the energy was used, contrary to other previous approaches, e.g., based on graduated nonconvexity \cite{BZ,niidmo98,nikolova99}.  The extension we provide in this paper  allows us now to similarly address problems, which are defined in any dimension, thanks to the appropriate handling of corresponding linear 
constraints and {\it very mild} smoothing.\\
{\it Thresholding algorithms} have by now a long history of successes, based on their extremely simple implementation, their statistical properties, and, in the iterative case, strong convergence guarantees. We retrace briefly some of the relevant developments, without the intention of providing an exhaustive mention of the many contributions in this area.
The terminology ``thresholding'' comes from image and signal processing literature, especially related to damping of wavelet coefficients in denoising problems, however the associated mathematical concept is the {\it Moreau proximity map} \cite{cowa05}, well-known from convex optimization. The statistical theory of thresholding has been pioneered by Dohono and Johnstone \cite{dojo94} in signal and image denoising and further and extensively explored in other work, e.g., \cite{chdelelu98}. {\it Iterative soft-thresholding algorithms} to numerically solve the minimization of convex energies, modelling inverse problems and formed by quadratic fidelity terms and $\ell_p$-norm penalties, for $p \geq 1$, have been first proposed in \cite{fino01}. Their strong convergence has been proven in the seminal work of Daubechies, Defrise, and De Mol \cite{DDD}. The recent theory of {\it compressed sensing}, i.e., the universal and nonadaptive compressed acquisition of data \cite{carota06,do06}, stimulated also the research 
of iterative thresholding algorithms for nonconvex penalty terms, such as the $\ell_p$-
quasi-norms for $0<p<1$. Variational and convergence properties of  {\it iterative firm-thresholding} algorithms, in particular the {\it iterative hard-thresholding}, have been recently studied in \cite{blda09,fora07}. Partially inspired by these latter achievements and the work of Nikolova \cite{nikolova} on the relationships between certain thresholding operators and discrete Mumford-Shah functionals, the results in \cite{fowa10} and in the present paper should be also considered as a contribution to the theory of thresholding algorithms in the new context of linearly constrained nonsmooth and nonconvex optimization.
\\

The paper is organized as follows. In Section \ref{algorithm} we define an appropriate concept of constrained critical points for certain classes of nonconvex functionals. We introduce then our new algorithm for the solution of nonsmooth and nonconvex minimization with linear constraints and we prove its convergence to critical points. Section \ref{applications} is addressed to the application of the general algorithm to the linearly constrained optimization problems of the type \eqref{examples}. In Section \ref{innerthr} we show how the core of the algorithm for free-discontinuity problems can actually be realized as a novel iterative soft-type thresholding algorithm. Section 4 is dedicated to numerical experiments, which demonstrate and confirm the theoretical findings, and, in particular, show how to tune the parameters of the algorithm. 
For ease of reading, we collect some of the technical results in a concluding Appendix.

\section{Linearly Constrained Nonsmooth and Nonconvex Minimization}\label{algorithm}
\subsection {Preliminaries and assumptions}
Let $\EE$ be an Euclidean space (that is a {\it finite dimensional} real Hilbert space) and $\J \colon \EE \to \R$ a lower semicontinuous functional which we assume to be bounded from below. Since we will be concerned with the search of critical points, without any loss of generality we shall suppose from now on that $\J(v) \ge 0$, for all $v \in \EE$. Let $\EE_1$ be another Euclidean space and we further consider a linear operator $A \colon \EE \to \EE_1$. Both the spaces $\EE$ and $\EE_1$ are endowed with an Euclidean norm, which we will denote in both cases by $\|\cdot\|$, since it will be always clear from the context in which space we are taking the norm. Dealing with finite dimensional spaces, it remains understood that the only notion of convergence that we will use is the strong convergence in norm, since weak and strong topologies are in this case equivalent. Let us also point out that most part of the analysis that we will carry out in the paper could be done in the more general framework of (
possibly 
infinite dimensional) Hilbert spaces; namely, the only point where finite dimensionality is actually needed, is to recover strong compactness in the proof of Theorem \ref{mainthm}.

We will deal in general with a nonconvex objective functional $\J$. When in the paper, as for instance in Section  \ref{malf} we will assume convexity of the objective function, it will be denoted by $\tilde \J$ to avoid any kind of confusion.
About the operator $A$, we  shall assume that $A$ has {\it nontrivial kernel}, and is {\it surjective}.  We shall denote by $A^*\colon \EE_1 \to \EE$ the adjoint operator of $A$. By our assumptions, for every $w \in \EE_1$ we have that there exists $\delta >0$ such that
\begin{equation}\label{inj}
\|A^*w\| \ge \delta \|w\|\,. 
\end{equation}

We consider $f \in \EE_1$ and we are concerned with the problem of finding constrained critical points of $\J$ on the affine space $\F(f):=\{v \in \EE: Av=f\}$. As usual in nonsmooth analysis, the notion of critical point is defined via the use of subdifferentiation.

\begin{definition}
Let $\EE$ be an Euclidean space, $\J \colon \EE \to \R$ a lower semicontinuous functional, and $v\in \EE$. We say that $\xi \in \EE' \simeq\EE$ belongs to the subdifferential $\partial \J(v)$ of $\J$ at $v$ if and only if\begin{equation}\label{1234}
 \liminf_{w\to v} \frac{\J(w)-(\J(v)+\langle \xi, w-v\rangle)}{\|w-v\|} \ge 0\,.
\end{equation}
\end{definition}

The subdifferential is single-valued precisely at (Fr\'echet) differentiability points, where it coincides with the differential, but can be in general multivalued, or even empty. It is well-known (see, for instance \cite[Chapter 1]{AGS}) that it is a closed convex set. In the special case of a convex functional $\tilde \J$, it is nonempty at every point and it can be shown (see again \cite[Proposition 1.4.4]{AGS}) that the definition of subdifferential given in \eqref{1234} and the one which is classical in convex nonsmooth analysis coincide, that is
\begin{equation}\label{12345}
 \xi \in \partial \tilde \J(v) \mbox{ if and only if } \tilde \J(w)-(\tilde \J(v)+\langle \xi, w-v\rangle) \ge 0
\end{equation}
for every $w\in \EE$. The symbol $\partial$ will be therefore used both in the convex and in the nonconvex case, since no ambiguity is possible.
In the case of a $C^1$ perturbation of a lower semincontinuous functional, that is $\J=\J_1+ \J_2$ where $\J_1$ is lower semicontinuous, and $\J_2$ is of class $C^1$, it follows from the definition that if $\partial \J_1(v)$ is nonempty, then $\partial \J(v) \neq \emptyset$ and the decomposition
\begin{equation}\label{2469}
\partial \J(v)= \partial \J_1(v)+ D\J_2(v)\,,
\end{equation}
holds true. Here $D$ denotes the Fr\'echet differential of $\J_2$ at $v$. In particular, $C^1$-perturbations of lower semicontinuous convex functionals have nonempty subdifferential at every point. We collect in the following Remark some useful properties of the subdifferential that will be employed in the sequel.

\begin{remark}\label{rem1}
If $\J$ is a $C^1$-perturbation of a convex function, one proves that the subdifferential enjoys the following closure property:
\begin{equation}\label{6480}
\xi_n \in \partial \J(v_n),\quad v_n \to v,\quad \xi_n \to \xi \mbox{ implies } \xi \in \partial \J(v) \hbox { and } \J(v_n) \to \J(v)\,.
\end{equation}

The subdifferential of a convex function $\tilde \J$ is known to be a monotone operator \cite{ET}, that is, for every $v$ and $w \in \EE$
\begin{equation}\label{monoton}
\xi \in \partial \tilde \J(v)\hbox{ and }\omega \in \partial \tilde \J(w)  \mbox{ implies } \langle \xi-\omega, v-w\rangle \ge 0\,.
\end{equation}
\end{remark}

We shall say that a function is {\it $\nu$-strongly convex} if a stronger form of \eqref{monoton} holds, that is there exists $\nu >0$ such that
\begin{equation}\label{monoton2}
\xi \in \partial \tilde \J(v)\hbox{ and }\omega \in \partial \tilde \J(w)  \mbox{ implies } \langle \xi-\omega, v-w\rangle \ge \nu\|v-w\|^2\,.
\end{equation}
It is well-known that this is equivalent to saying that $\tilde \J(\cdot)-\frac \nu2 \|\cdot\|^2$ is convex.

We are now ready to recall the definition of critical point. 

\begin{definition}
Let $\EE$ be an Euclidean space, $\J \colon \EE \to \R$ a lower semicontinuous functional, and $v\in \EE$. We say that $v$ is a critical point of $\J$ if 
$$
0 \in \partial \J(v)\,.
$$
\end{definition}

In the convex case this condition is sufficient to assure global minimality of $v$, otherwise it is only a necessary condition for local minimality.

In the following definition of constrained critical point the usual shorthand $\J(w+~\cdot)$ is used to denote the functional $\xi \mapsto \J(w+\xi)$.
\begin{definition}
Given a linear operator $A \colon \EE \to \EE_1$ with nontrivial kernel, and $f \in \EE_1$, we say that $w$ is a critical point of $\J$ on the affine space $\F(f)=\{v \in \EE: Av=f\}$ if $Aw=f$ and $0$ is a critical point for the restriction to $\ker A$ of the functional $\J(w+\cdot)$.\end{definition}

For $\J$ being a $C^1$-perturbation of a convex function (in particular, with nonempty subdifferential at every point), the nonsmooth version of Lagrange multiplier Theorem assures that $w$ is a critical point of $\J$ on the affine space $\{v \in \EE: Av=f\}$ if and only if $Aw=f$ and
\begin{equation}\label{6529}
\partial \J(w) \cap \ran(A^*)\neq \emptyset\,,
\end{equation}
where $\ran(A^*)$ is the range of the operator $A^*$, which is known to be the orthogonal complement of $\ker A$ in $\EE$.

From now, about the function $\J$, we will make the following more specific assumptions:
\begin{itemize}
 \item [(A1)] $\J$ is $\omega$-semi-convex, that is there exists $\omega >0$ such that $\J(\cdot)+\omega \|\cdot\|^2$ is convex;
 \item [(A2)] the subdifferential of $\J$ satisfies the following growth condition: there exist two nonnegative numbers $K$, $L$ such that, for every $v \in \EE$ and $\xi \in \partial \J(v)$
 \begin{equation}\label{strana}
  \|\xi\| \le K \J(v)+ L\,.
 \end{equation}
\end{itemize}

\begin{remark}\label{note}
(a) We observe that condition (A1) is in fact met, for instance, by any $C^1$ function in finite dimension with piecewise continuous and bounded second derivatives.   However, let us stress that, conversely, $\omega$-semi-convexity does not give any information on the smoothness of the function, other than local Lispchitzianity, hence, in finite dimension, its Fr\'echet-differentiability almost everywhere, by Rademacher's Theorem.
We also recall that an $\omega$-semi-convex function is a $C^1$-perturbation of a convex function, therefore it has nonempty (and locally bounded) subdifferential at every point.  If the subdifferential is uniformly bounded, then \eqref{strana} is trivially satisfied.

(b) As just recalled,  an $\omega$-semi-convex function in finite dimension has a Fr\'echet differential almost everywhere, and, if \eqref{strana} is satisfied only at points of differentiability, then it holds everywhere. This is true since it can be shown that the Fr\'echet subdifferential is contained in the so-called Clarke subdifferential, which is known to be at every $v \in \EE$ the convex hull of limit points of differentials of $\J$ along sequences $v_n \to v$ (for these notions, see for instance \cite[Chapter 2]{Cla}).
Therefore one needs not to calculate the subdifferential of $\J$ at non-differentiability points (which is in general quite a hard task) to check if the hypothesis is satisfied everywhere.
\end{remark}

Given $\omega >0$, and $u \in \EE$ we will denote 
\begin{equation}\label{1925}
\J_{\omega,u}(v):=\J(v)+\omega\|v-u\|^2
\end{equation}
Notice that $\J_{\omega,u}$ is coercive whenever $\J$ is bounded from below. We observe that, if $\J$ satisfies (A1) we can always assume that $\omega$ is chosen in such a way that $\J_{\omega,u}$ is also $\nu$-strongly convex with $\nu$ depending on $\J$ and $\omega$, but not on $u$. Analogously, if (A1) and (A2) are satisfied, by using \eqref{2469} it is easy to see that $\J_{\omega,u}$ satisfies \eqref{strana} with two constants $\tilde K, \tilde L$ depending again on $\J$ and $\omega$, but not on $u$.

\subsection{The augmented Lagrangian algorithm in the convex case}\label{malf}
We now recall some basic facts about augmented Lagrangian iterations for constrained minimization of convex functionals. 
Here, we are given a {\it coercive convex functional} $\tilde \J$ and, given two arbitrary $v_0 \in \EE$, $q_0 \in \EE_1$ , for every $k \in \N$, $k \ge 1$, we define:
\begin{equation}\label{bregman}
\begin{cases}
v_k \in \hbox{arg min}_{v \in \EE} (\tilde \J(v)-\langle q_k, Av \rangle + \lambda \|Av-f\|^2) \,.\\
q_k = q_{k-1}+ 2\lambda (f-Av_k)\,.
\end{cases}
\end{equation}

Convergence of the algorithm has been proved in \cite{Bre1}, where it was called {\it Bregman iteration}, and, since it is equivalent to the {\it Augmented Lagrangian Method} \cite{ItoKun}, also in \cite{FrSch}.
Precisely it has been shown that $\|Av_k-f\|$ decreases to $0$ as $k$ tends to $+\infty$, that the sequence $v_k$ is compact and any limit point is a global minimum of $\tilde \J$ under the constraint $Av=f$.
Moreover, for every $k\ge 1$, $A^*q_k \in \partial \J(v_k)$.
When $\tilde \J$ is $\nu$-strongly convex for some $\nu >0$ we have also a quantitative estimate of the convergence of $v_k$ to the unique (due to strict convexity) minimizer of the problem. We give a precise statement and a proof of this additional property, as it will be very useful later in the nonconvex case as well. 

\begin{proposition}\label{stima}
Assume that $\tilde \J$ is $\nu$-strongly convex, let $v_k$ and $q_k$ the sequences generated by \eqref{bregman}, and let $\bar v$ the unique global minimizer of $\tilde \J$ on the affine space $\{v \in \EE: Av=f\}$. Then:
\begin{itemize}
 \item [(i)] $(\|Av_k-f\|)_{k \in \mathbb N}$ is a decreasing sequence;
 \item [(ii)] $\lim_{k \to +\infty}\|Av_k-f\|=0$;
\item [(iii)] $\|v_k-\bar v\|^2 \le \frac1\nu\|q_0- \bar q\|\,\|Av_k-f\|$, for all $k \in \mathbb N$,
\end{itemize}
for every $\bar q \in \EE_1$ such that $A^*\bar q \in \partial \tilde \J(\bar u)$.
\end{proposition}

\begin{proof}
Properties (i) and (ii) are proved in \cite{Bre1}. For the property (iii), we first observe that such a $\bar q$ surely exists by \eqref{6529}. We define for all $k \ge 1$ the discrepancy $\Delta q_k:= q_k-\bar q$, and we prove that $\|\Delta q_k\|$ is decreasing. We actually have, by elementary computations and using \eqref{bregman}, that
\begin{eqnarray}\label{zuip}
&\displaystyle
\|\Delta q_k\|^2-\|\Delta q_{k-1}\|^2 \le 2\langle q_k-q_{k-1}, q_k-\bar q\rangle = \\
&\displaystyle
4\lambda\langle f-Av_k, q_k-\bar q\rangle=
4\lambda\langle \bar v-v_k, A^*q_k-A^*\bar q\rangle\,.
\end{eqnarray}
Since $A^*q_k \in \partial \tilde \J(v_k)$ and $A^*\bar q \in \partial \tilde \J(\bar u)$, the last term in the inequality is nonpositive by \eqref{monoton}, therefore the claim follows. In particular
\begin{equation}\label{monb}
\|q_k- \bar q\| \le \|q_0- \bar q\|\,,
\end{equation}
for all $k \geq 1$.
Now, by \eqref{monoton2}, we have also
$$
\nu \|v_k-\bar v\|^2 \le  \langle A^*q_k-A^*\bar q, v_k- \bar v\rangle= \langle q_k-\bar q, Av_k- f\rangle\,,
$$
so that we conclude by the Cauchy-Schwarz inequality and \eqref{monb}.
\end{proof}

When $\tilde \J$ is the function $\J_{\omega,u}$ defined by \eqref{1925}, with an appropriate choice of $\omega$, by the previous result, \eqref{inj}, and \eqref{strana}, we get the following corollary, whose rather immediate proof is therefore omitted.
\begin{corollary}\label{caz}
Consider the function $\J_{\omega,u}$ defined by \eqref{1925}, where $\omega$ is chosen in such a way that $\J_{\omega,u}$ is $\nu$-strongly convex with $\nu$ not depending on $u$.
Let $\bar v_u$ be the unique global minimizer of $\J_{\omega,u}$ on the affine space $\{v \in \EE: Av=f\}$. Then there exist two positive constants $C_1$ and $C_2$ depending on $A^*$, $\J$, and $\omega$, but not on $u$, such that
\begin{equation}\label{1785}
\|v_{k,u}-\bar v_u\|^2 \le \left [ C_1 (1+ \|q_0\|)+ C_2 \J_{\omega,u}(\bar v_u) \right ]\,\|Av_{k,u}-f\|\,,
\end{equation}
where $v_{k,u}:=v_k$ is defined accordingly to \eqref{bregman} for $\tilde \J=\J_{\omega,u}$.
\end{corollary}

\subsection{The algorithm in the nonconvex case}
We now present the new algorithm for linearly constrained nonsmooth and nonconvex minimization, and discuss its convergence properties.
We pick initial $v_{(0,0)} \in \EE$ and $q_{(0,0)} \in \EE_1$. Notice that there is no restriction to any specific neighborhood for the choice of the initial iteration. For a fixed scaling parameter $\lambda >0$, and an {\it adaptively chosen} sequence of integers $(L_\ell)_{\ell \in \mathbb N}$, for every integer $\ell \ge 1$ we set (with the convention $L_0=0$):
\begin{equation}\label{ncbregman}
\begin{cases}
v_{(\ell,0)}=v_{\ell-1}:=v_{(\ell-1,L_{\ell-1})}\quad q_{(\ell,0)}=q_{\ell-1}:=q_{(\ell-1,L_{\ell-1})}\\
v_{(\ell,k)} = \hbox{arg min}_{v \in \EE} \,\big( \J_{\omega, v_{\ell-1}}(v)-\langle q_{(\ell,k-1)}, Av \rangle + \lambda \|Av-f\|^2\big)\,,\quad k=1,\dots,L_\ell \\
q_{(\ell,k)} = q_{(\ell,k-1)}+ 2\lambda (f-Av_{(\ell,k)})\,.
\end{cases}
\end{equation}
Here, thanks to condition (A1), $\omega$ is chosen in such a way that $\J_{\omega, v_{\ell-1}}$ is $\nu$-strongly convex, with $\nu$ independent of $v_{\ell-1}$, and the {\it finite} number of inner iterates $L_\ell$ is defined by the condition
\begin{equation}\label{adpt}
(1+\|q_{\ell-1}\|)\|Av_{(\ell,L_\ell)}-f\|\le \frac{1}{\ell^\alpha}\,,
\end{equation}
for a given parameter $\alpha >1$.

Since the inner loops are simply the augmented Lagrangian iterations for the functional $\J_{\omega, v_{\ell-1}}$, by Proposition \ref{stima} (ii) and \eqref{monb} such an integer $L_\ell$ always exists. We also remark that by construction, for every $\ell \ge 1$ and $k=1, \dots, L_\ell$, we have 
\begin{equation}\label{inclu}
A^*q_{(\ell,k)} \in \partial \J_{\omega, v_{\ell-1}}(v_{(\ell,k)})\,.
\end{equation}
Moreover, for every $\ell \ge 1$, again by Proposition \ref{stima}, $\|Av_{(\ell,k)}-f\|$ is nonincreasing in $k$.

Let us also remark that \eqref{ncbregman}, which can also be viewed as an implementation of an implicit
gradient descent with step $\frac 1\omega$, is actually a natural generalization of \eqref{bregman}. Indeed, if $\mathcal J =\tilde \J$ were actually convex, we could in fact choose $\omega=0$, and \eqref{ncbregman} would simply reduce to \eqref{bregman}.
\subsection{Analysis of convergence}

We now want to analyse the convergence properties of the algorithm defined by \eqref{ncbregman}. To do that we will use the following basic calculus lemma.
\begin{lemma}\label{cauchy}
Let $(a_\ell)_{\ell \in \mathbb N}$ a sequence of positive numbers, and let $(\delta_\ell)_{\ell \in \mathbb N}$ a positive decreasing sequence such that
$$
\sum_{\ell=0}^\infty \delta_\ell <+\infty\,.
$$
If $a_\ell$ satisfies for every $\ell$ the inequality
\begin{equation}
a_\ell \le (1+ \delta_{\ell-1})a_{\ell-1}+\delta_{\ell-1}\,,
\label{recurr}
\end{equation}
then $(a_\ell)_{\ell \in \mathbb N}$ is a convergent sequence.
\end{lemma}
\begin{proof}
By the recurrence relation \eqref{recurr} we deduce
\begin{equation}
a_\ell \leq \left [\prod_{k=0}^{\ell-1} (1+ \delta_k) \right] a_0 + \sum_{\ell'=0}^{\ell-1}  \left [\prod_{k=\ell'+1}^{\ell-1} (1+ \delta_k) \right ] \delta_{\ell'}. \label{recurr2}
\end{equation}
Notice that 
\begin{eqnarray}
\log \left [\prod_{k=0}^{\infty} (1+ \delta_k) \right] &=& \sum_{k=0}^\infty \log (1+ \delta_k) \nonumber \\
&=& \sum_{k=0}^\infty \left ( \delta_k -\frac{1}{2 \xi_k} \delta_k^2 \right)< \infty,
\end{eqnarray}
for suitable $\xi_k \in (1,1+\delta_k)$,  for $k \in \mathbb N$, hence
$$
\prod_{k=0}^{\infty} (1+ \delta_k)  < \infty,
$$
and, together with \eqref{recurr2}, we deduce that $(a_\ell)_{\ell \in \mathbb N}$ is actually uniformly  bounded. Now, again by the recurrence relation \eqref{recurr}, for $k' \leq k$, we obtain
$$
a_k = a_{k'} +\sum_{\ell=k'+1}^k (a_\ell - a_{\ell-1}) \leq a_{k'}+\sum_{\ell=k'+1}^k \delta_{\ell-1} a_{\ell-1}+\sum_{\ell=k'+1}^k  \delta_{\ell-1}.
$$
Taking first the $\limsup$ as $k \to +\infty$ and then the $\liminf$ as $k'\to +\infty$ in the previous inequality, we conclude from the boundedness of $(a_\ell)_{\ell \in \mathbb N}$ and the convergence of the series $\sum_{\ell=0}^\infty \delta_\ell$ that $\limsup_{k \to +\infty}a_k\le \liminf_{k' \to +\infty}a_{k'}$, which implies the conclusion.
\end{proof}

In the following theorem we analyse the convergence properties of the proposed algorithm.
\begin{theorem}\label{chiave}
Assume that $\J$ satisfies (A1) and (A2), and let $(v_\ell)_{\ell \in \mathbb N}$ be the sequence generated by \eqref{ncbregman}. Then,
\begin{itemize}
 \item[(a)] $(Av_\ell-f) \to 0$ as $\ell \to \infty$;
 \item[(b)] $(v_\ell-v_{\ell-1}) \to 0$ as $\ell \to \infty$.
\end{itemize}
If in addition $\J$ is coercive on the affine space $\{v \in \EE: Av=f\}$, then $v_\ell$ is bounded and $(\J(v_\ell))_{\ell \in \mathbb N}$ is a convergent sequence. More in general, if $\J$ only satisfies (A1) and (A2), the implication
\begin{equation}\label{1378}
\mbox{if }(v_\ell)_{\ell \in \mathbb N} \hbox{ is a bounded sequence,}\mbox{ then } (\J(v_\ell))_{\ell \in \mathbb N} \hbox{ is convergent}
\end{equation}
holds.
\end{theorem}

\begin{proof}
Part (a) of the statement is a direct consequence of the construction of $v_\ell$ and Proposition \ref{stima} (ii).
We now set for every $\ell$
\begin{equation}\label{barv}
\bar v_\ell:=\hbox{ arg min}_{Av=f} \J_{\omega, v_{\ell-1}}(v)\,.
\end{equation}
Notice that by definition $\bar v_\ell$ coincides with the element $\bar v_u$ considered in Corollary \ref{caz} when $u=v_{\ell-1}$. Similarly the element $v_\ell$ given by algorithm \eqref{ncbregman} coincides with the element $v_{k,u}$ considered in Corollary \ref{caz} when $k=L_\ell$ and $u=v_{\ell-1}$. Therefore \eqref{1785} with $q_0=q_{(\ell,0)}$ and \eqref{adpt} imply there exist two positive constants $C_1$ and $C_2$ independent of $\ell$, such that
\begin{equation}\label{17852}
\|v_\ell-\bar v_\ell\|^2 \le [C_1+ C_2 \J_{\omega, v_{\ell-1}}(\bar v_\ell)]\frac{1}{\ell^\alpha}\,.
\end{equation}
By this latter estimate and the minimality of $\bar v_{\ell+1}$ we get
\begin{eqnarray}
\J_{\omega, v_\ell}(\bar v_{\ell+1})&=&\J(\bar v_{\ell+1})+\omega \|v_\ell-\bar v_{\ell+1}\|^2  \nonumber \\
&\le& \J(\bar v_\ell)+\omega \|v_\ell-\bar v_\ell\|^2 \le \J(\bar v_\ell)+\frac{C_1\omega}{\ell^\alpha}+\frac{C_2\omega}{\ell^\alpha}\J_{\omega, v_{\ell-1}}(\bar v_\ell)\label{2938}\\
&\le& \frac{C_1\omega}{\ell^\alpha}+\Big(1+\frac{C_2\omega}{\ell^\alpha}\Big)\J_{\omega, v_{\ell-1}}(\bar v_\ell)\, \nonumber.
\end{eqnarray}
By Lemma \ref{cauchy} we eventually deduce that $(\J_{\omega, v_{\ell-1}}(\bar v_\ell))_{\ell \in \mathbb N}$ is a convergent sequence, in particular it is bounded. Therefore, there exists a constant $C$ independent of $\ell$ such that, by \eqref{17852}, 
\begin{equation}\label{17853}
\|v_{\ell+1}-\bar v_{\ell+1}\|^2 \le \frac{C}{(\ell+1)^\alpha}\,,
\end{equation}
and, by \eqref{2938}, we have also
\begin{equation}\label{29381}
\J(\bar v_{\ell+1})\le \J(\bar v_{\ell+1})+\omega \|v_\ell-\bar v_{\ell+1}\|^2 \le \J(\bar v_\ell)+\frac{C}{\ell^\alpha}\,.
\end{equation}
Again Lemma \ref{cauchy} entails now that
\begin{equation}\label{456}
\J(\bar v_\ell)\hbox{ is a convergent sequence}\,,
\end{equation}
so that, by \eqref{29381} we get that $(v_\ell-\bar v_{\ell+1}) \to 0$ as $\ell$ goes to $+\infty$, and this vanishing convergence, combined with \eqref{17853}, gives part (b) of the statement.

Being $\J$ locally Lipschitz as it is an $\omega$-semi-convex function, if $v_\ell$ is uniformly bounded, by \eqref{17853} and \eqref{456} we immediately conclude that $(\J(v_\ell))_{\ell \in \mathbb N}$ is a convergent sequence. Moreover, if $\J$ is coercive on the affine space $\{v \in \EE: Av=f\}$, then $\bar v_\ell$ is bounded by \eqref{456}, and so is also $(v_\ell)_{\ell \in \mathbb N}$ by \eqref{17853}, as required.
\end{proof}

As a consequence we get our main result of this section. Whenever $v_\ell$ is bounded, every cluster point is a constrained critical point of $\J$ on the affine space $\{v \in \EE: Av=f\}$. We again recall that boundedness of $v_\ell$ is guaranteed by Theorem \ref{chiave} when $\J$ is assumed to be coercive on the above affine space.

\begin{theorem}\label{mainthm}
Assume that $\J$ satisfies (A1) and (A2), and let $(v_\ell)_{\ell \in \mathbb N}$ be the sequence generated by \eqref{ncbregman}. If $(v_\ell)_{\ell \in \mathbb N}$ is bounded, every of its limit points is a constrained critical point of $\J$ on the affine space $\{v \in \EE: Av=f\}$. 
\end{theorem}

\begin{proof}
Let $(q_\ell)_{\ell \in \mathbb N}$ be the sequence defined by \eqref{ncbregman}, and let $p_\ell:=A^*q_\ell$, and $\hat p_\ell:=p_\ell-2\omega(v_\ell-v_{\ell-1})$. By \eqref{2469} and \eqref{inclu}, we have 
\begin{equation}\label{meno2}
\hat p_\ell \in \partial \J(v_\ell)\,,
\end{equation}
and by the boundedness of $(v_\ell)_{\ell \in \mathbb N}$, \eqref{1378}, and (A2), we then get that $\hat p_\ell$ is bounded too.
By Theorem \ref{chiave}, part (b), we deduce that $p_\ell-\hat p_\ell \to 0$, which in particular gives
\begin{equation}\label{meno1}
\lim_{\ell \to +\infty}{\rm dist}(\hat p_\ell\,, \ran(A^*))=0\,.
\end{equation}
Now, if a subsequence $v_{\ell_j} \to v \in \EE$, possibly taking a further subsequence we may assume that $\hat p_{\ell_j} \to \hat p \in \partial \J(v)$, where the last inclusion follows from \eqref{6480} and \eqref{meno2}. Moreover, since in finite dimension $\ran(A^*)$ is closed, by \eqref{meno1} $\hat p \in \ran(A^*)$. Since $Av=f$ by part (a) of Theorem \ref{chiave}, \eqref{6529} yields now the desired conclusion.
\end{proof}

\section{Examples of Significant Applications}\label{applications}

We consider again two finite dimensional Euclidean spaces $\EE$ and $\EE_1$ and a surjective linear constraint map $A: \EE \to \EE_1$. In addition we consider another finite dimensional Euclidean space $\mathcal K$ and a linear operator $T: \EE \to \mathcal K$. Again, to ease the notation, we indicate with $\| \cdot\|$ the Euclidian norms on $\EE$, $\EE_1$, or $\mathcal K$ indifferently, as they can be subsumed from the context where they are applied.
For fixed $g \in \mathcal K$ and $f \in \EE_1$,  in the following we consider general nonsmooth and nonconvex functionals of the type
\begin{equation} \label{specfunc}
\J(v) = \| T v - g \|^2 + \gamma \sum_{k=1}^m  U_k(v_k),
\end{equation} 
of which we seek the critical points, subject to a linear constraint $A v = f$, where $\gamma >0$ is a positive regularization parameter. Here $(v_k)_{k=1}^m$ are the components  of the vector $v$ with respect to a fixed basis in the space $\EE$ of dimension $m = \dim(\EE)$,
and $U_k:\mathbb R \to \mathbb R_+$, for $k=1,\dots,m$, are scalar nonconvex maps. \\

With the intention of demonstrating the very broad  impact of Algorithm \eqref{ncbregman}, in this section we would like to present a (incomplete!) list of significant models, mainly inspired by image processing and continuum mechanics, where the Algorithm \eqref{ncbregman} is already directly and robustly  used. In particular we shall show that after discretization such models fit the general optimization problem \eqref{specfunc} with specific choices of the maps $U_k$. We further discuss the condition of applicability of Algorithm \eqref{ncbregman} in each of them.

\subsection{Free-discontinuity problems}

The terminology `free-discontinuity problem' was introduced by De Giorgi \cite{dg91} to indicate a class of variational problems which consist in the minimization of a functional, involving both volume and surface energies, depending on a closed set $K \subset \mathbb R^d$, and a function $u$ on $\mathbb R^d$ usually smooth outside of $K$. In particular, 
\begin{itemize}
\item $K$ is not fixed a priori and is an unknown of the problem;
\item $K$ is not a boundary in general, but a free-surface inside the domain of the problem.
\end{itemize}

\subsubsection{The  Mumford-Shah functional in image processing}\label{MSfunc}

\noindent The best-known example of a free-discontinuity problem is the one modelled by the so-called Mumford-Shah functional \cite{MS89}, which is defined by
$$
\mathcal J(u,K):= \int_{\Omega \setminus K} \left [ | \nabla u(x) |^2 + \alpha ( u(x) - g(x) )^2 \right ] dx +\beta \mathcal{H}^{d-1}(K \cap \Omega).
$$
The set $\Omega$ is a bounded open subset of $\mathbb{R}^d$, $\alpha, \beta >0$ are fixed constants, and $g \in L^\infty(\Omega)$. Here $\mathcal H^N$ denotes the $N$-dimensional Hausdorff measure. Inspired by image processing applications the dimension of the underlying Euclidean space $\mathbb R^d$ shall  be $d=2$, although in principle the analysis can be conducted in any dimension.  In fact, in the context of visual analysis, $g$ is a given noisy image that we want to approximate by the minimizing function $u \in W^{1,2}(\Omega \setminus K)$; the set $K$ is simultaneously used in order to {\it segment} the image into connected components.  For a broad overview on free-discontinuity problems, their analysis, and applications, we refer the reader to \cite{AFP}.
\\

In fact, the Mumford-Shah functional is the continuous version of a previous discrete formulation of the image segmentation problem proposed by Geman and Geman in \cite{GG}; see also the work of Blake and Zisserman in \cite{BZ}. Let us recall this discrete approach.
Let $d=2$ (as for image processing problems),  $\Omega = [0,1]^2$, and let $u_{i,j}=u(h i,h j)$ be a discrete function defined on $\Omega_h:=\Omega \cap h \mathbb{Z}^2$, for $h>0$.
Define $W_{r}^2 ( t ) := \min \left \{ t^2, r^2 \right \}$, $r>0$, to be the {\it truncated quadratic potential}, and
\begin{eqnarray}
\mathcal J_{h}(u ) &:=& h^2 \sum_{(h i, h j) \in \Omega_h} W_{\sqrt{\frac{\beta}{h}}}^2 \left ( \frac{u_{i+1,j} - u_{i,j}}{h} \right ) \nonumber \\
&+& h^2 \sum_{ ( h i, h j) \in \Omega_h} W_{\sqrt{\frac{\beta}{h}}}^2  \left ( \frac{u_{i,j+1} - u_{i,j}}{h} \right )\nonumber \\
&+& \alpha h^2 \sum_{(h i,h j) \in \Omega_h} ( u_{i,j} - g_{i,j})^2.
\label{discrete}
\end{eqnarray}
We shall now reformulate the minimization of this finite dimensional discrete problem into a linearly constrained minimization of a nonconvex functional of the discrete derivatives.   For this purpose, we consider the derivative matrix $D_h : \mathbb R^{n^2} \to \mathbb R^{2n(n-1)}$ that maps the vector $(u_{j + (i-1) n}) := (u_{i,j})$ to the vector composed of the finite differences in the horizontal and vertical directions $u_x$ and $u_y$ respectively, given by
$$
D_h u := \left [ 
\begin{array}{l}
u_x\\
u_y
\end{array}
\right ], \quad \left \{ \begin{array}{ll} (u_x)_{j + n(i-1) }:= (u_x)_{i,j}:= \frac{u_{i+1,j} - u_{i,j}}{h}, i=1,\dots,n-1, j=1,\dots,n\\ (u_y)_{j + (n-1)(i-1) }:= (u_y)_{i,j}:= \frac{u_{i,j+1} - u_{i,j}}{h}, i=1,\dots,n, j=1,\dots,n-1
\end{array} \right. .
$$
{ Note that its range $\ran(D_h) \subset \mathbb{R}^{2n(n-1)}$ is a $(n^2-1)$-dimensional subspace} because $D_h c = 0$ for constant vectors $c \in \mathbb{R}^{n^2}$.  It is not difficult to show the representation of any vector $u \in \mathbb R^{n^2}$ in terms of the following differentiation-integration formula, given by
$$
 u =D^\dagger_h D_h u + c,
$$
where $D_h^\dagger$ is the pseudo-inverse matrix of $D_h$ (in the Moore-Penrose sense); note that $D^\dagger_h$ maps $\ran(D_h)$ injectively into $\mathbb R^{n^2}$.  Also, $c$ is a constant vector that depends on $u$, and the values of its entries coincide with the mean value $h^2\sum_{ (hi,hj) \in \Omega_h} u_{i,j}$ of $u$. Therefore, any vector $u$ is uniquely identified by the pair $(D_h u , c)$.
\\

\noindent  
Since constant vectors comprise the null space of $D_h$, the orthogonality relation 
\begin{equation}\label{orthog}
\langle D^\dagger_h D_h u,c\rangle =0
\end{equation} holds for any vector $u$ and any constant  vector $c$. Here the scalar product $\langle u, u' \rangle = \sum_{ (hi,hj) \in \Omega_h} u_{i,j} u_{i,j}'$ is the standard Euclidean scalar product on $\mathbb R^{n^2}$, which induces the Euclidean norm $\| \cdot \|$.
\\

Using the orthogonality property \eqref{orthog}, denoting the mean value of $g$ by $c_g$, we have that
\begin{eqnarray}
\| u - g \|^2 &=& \| D^\dagger_h D_h u -  D^\dagger_h D_h g + ( c - c_g )\|^2 \nonumber \\
&=&    \| D^\dagger_h D_h u -  D^\dagger_h D_h g\|^2 + \| c - c_g \|^2 \label{orthog2}
\end{eqnarray}
Hence, with a slight abuse of notation, we can reformulate the original discrete functional $\eqref{discrete}$ in terms of derivatives, and mean values, by
\begin{eqnarray*}
\mathcal J_{h}(v,c) &=& h^2 \left [ {\alpha} \| D_h^\dagger v  - \tilde g  \|^2 + {\alpha} \| c - c_{g} \|^2 + \sum_{i,j} \min \left \{|v_{i,j}|^2, \frac{\beta}{h}  \right \} \right ].
\label{2Dunconstrained}
\end{eqnarray*}
where $v = D_h u \in \mathbb{R}^{2n(n-1)}$, and $\tilde g =  D^{\dagger}_h D_h g \in \mathbb{R}^{n^2}$. Of course $c = c_g$ is  assumed at any minimizer $u$, since the corresponding term in $\mathcal J_{h}$ does not depend on $v$.  However, in order to minimize only over vectors in $\mathbb R^{2n(n-1)}$ that are derivatives of vectors in $\mathbb R^{n^2}$, we must minimize $\mathcal J_{h}(v,c)$ subject to the constraint $(D_h D^{\dagger}_h   - I) v= 0$, and such $2n(n-1)$ linearly independent constraints  actually correspond to a discrete ${\rm curl}$-free condition on the vector $v$. 
\\
To summarize, we arrive at the following constrained optimization problem:

\begin{eqnarray}
&&
 \left\{ \begin{array}{llll}
\textrm{Minimize} &  \mathcal J_{h}(v) = h^2 \big[ {\alpha} \| T v  - \tilde g \|^2+  \sum_{i,j} W^2_{\sqrt{\frac{\beta}{h}}}(v_{i,j}) \big].
\\
\\
 \textrm{subject to} & A v = 0, \\\end{array}
\right.
\label{MS2d}
\end{eqnarray}
for $T = D_h^\dagger$ and { $A = {I}- D_h D^{\dagger}_h $.}  Actually the explicit use of the pseudo-inverse matrix $D^\dagger$ is only needed for determining the operator $T$, while the linear constraint $Av =0$ is equivalent to a discrete $\operatorname{curl}$-free condition on the vectors $v$, see \cite{fowa10} for details; therefore it can simply be expressed in terms of a sparse linear system corresponding to the discretization of the $\operatorname{curl}$ operator.
Once the minimal derivative vector $v$ is computed, we can assemble the minimal $u$ by incorporating the mean value $c_g$ of $g$ as follows:
$$
 u = D_h^\dagger v + c_g.
$$
We stress that when $v$ is curl-free, a primitive $u$ can be easily recovered, up to an arbitrary constant, by performing a line integration, so again this process does not require the explicit form of $D_h^\dagger$, see the details
of \eqref{lineint} and \eqref{lineint2} in Subsection \ref{brittle} below. Notice that the objective functional in the optimization problem \eqref{MS2d} is precisely of the form \eqref{specfunc}, with the maps $U_k =  h^2 W^2_{\sqrt{\frac{\beta}{h}}}$ for all $k=(i,j)$.
As we shall discuss in Section \ref{issuesapp}, the functional $\mathcal J_{h}$ in \eqref{MS2d} is in general neither coercive nor $\omega$-semi-convex as required by the conditions
of applicability of Algorithm \eqref{ncbregman}. Nevertheless we shall see in Section \ref{smoothapp} that a mild regularization will allow us to treat efficiently also problems as \eqref{MS2d}. We provide in  Remark \ref{remfin} below a possible guideline on the efficient implementation of the pseudoinverse matrix $D_h^\dagger$ and its adjoint $(D_h^\dagger)^*$, as it appears in the iterations of the inner loop of Algorithm \eqref{ncbregman} when applied to the minimization of the Mumford-Shah functional. Since it is genuinely a numerical issue, which also does not appear in the other possible applications we are going to discuss (see Subsections \ref{brittle} and \ref{cohesive} below), the task of implementing Remark 4.1 goes beyond the scope of this paper.

\subsubsection{Quasi-static evolution of brittle fractures}\label{brittle}

Beside static models such as the Mumford-Shah functional minimization for image deblurring and denoising, quasi-static evolutions of elastic bodies through minimizers 
of free-discontinuity energies are of great relevance in continuum mechanics.  In this modeling, the crack-free reference configuration of a linearly elastic body is denoted by 
$\Omega$. The set $\Omega \subset \mathbb R^d$ is taken to be an open, bounded and connected domain with Lipschitz boundary $\partial \Omega$, for instance $\Omega = (0,1)^d$. We consider an energy containing again 
bulk and surface terms
$$
\mathcal J(u, K) :=  \int_{\Omega \setminus K} | \nabla u(x) |^2 dx
+\beta \mathcal{H}^{d-1}(K \cap \Omega).
$$
The energy functional $\mathcal J(u,K)$ reflects Griffith's principle that, to create a crack one has to spend
an amount of elastic energy that is proportional to the area of the crack created \cite{gr21}.
The configuration of the body evolves in time under the action of a varying load $g(t)$, which
is applied on an open subset $\Omega_D \subset \Omega$ of positive $d$-dimensional Lebesgue measure. We assume
that $g \in L^\infty ([0, T] ; W^{1,\infty}(\Omega)) \cap W^{1,1} ([0, T] ; H^1 (\Omega))$, and we define the admissible
displacements consistent to the actual load
\begin{equation}\label{loadconstr}
\mathcal A(g(t)) := \{u \in SBV(\Omega) : u|_{\Omega_D} = g(t)|_{\Omega_D} \},
\end{equation}
where $SBV(\Omega)$ is the space of special bounded variation functions \cite{AFP}.
The model of brittle fracture discrete time evolution proposed by Francfort and Marigo \cite{frma98} is described as follows:
Let $0 = t_0 < t_1 < \dots < t_F = T$ be a discretization of the time interval $[0, T ]$, with $\Delta t :=
\max \{t_k - t_{k-1} : k = 1, \dots, F \}$. Given an initial crack $K(0) = \Gamma_D$, we seek $(u(t_k ), K(t_k ))$, $k = 1,\dots, F$ , 
such that
\begin{equation}\label{instmin}
(u(t_k ),K(t_k )) = \arg \min_{ u \in \mathcal A(g(t_k))} \mathcal J(u, K \cup K(t_{k-1})).
\end{equation}

Following a similar discretization in space for $d=2$ as done for the Mumford-Shah functional
in Section \ref{MSfunc}, let  $u_{i,j}=u(h i,h j)$ be a discrete function defined on $\Omega_h:=\Omega \cap h \mathbb{Z}^2$, for $h>0$.
Define
\begin{eqnarray}
\mathcal J_{h}(u ) &:=& h^2 \sum_{(h i, h j) \in \Omega_h \setminus K_h} W_{\sqrt{\frac{\beta}{h}}}^2 \left ( \frac{u_{i+1,j} - u_{i,j}}{h} \right ) \nonumber \\
&+& h^2 \sum_{ ( h i, h j) \in \Omega_h \setminus K_h} W_{\sqrt{\frac{\beta}{h}}}^2  \left ( \frac{u_{i,j+1} - u_{i,j}}{h} \right ),
\label{discrete2}
\end{eqnarray}
where $K_h$ is the current fracture. Up to considering appropriate domain decompositions and without loss of generality we can assume $K_h = \emptyset$ and
$\Omega_h$ be the reference domain for the optimization. Defining
$$
v_{i,j} = \left ( \underbrace{\frac{u_{i+1,j} - u_{i,j}}{h}}_{:=(v_{i,j})_1}  ,  \underbrace{\frac{u_{i,j+1} - u_{i,j}}{h}}_{:=(v_{i,j})_2} \right),
$$
we have that $v=(v_{i,j})_{ (h i, h j) \in \Omega_h}$ fulfills again a $\operatorname{curl}$-free condition as well as several linear constraints given by the 
discretization of the compatibility condition $u \in \mathcal A(g(t_k))$. In particular, if we assume that $g$ is locally constant with a jump at a given
interface $\Gamma_D \subset \Omega$, we may  have, as in the example shown in Figure \ref{domainsfr}, that
\begin{equation}
\left \{ \begin{array}{ll}
v_{i,j}=0, & \mbox{ for all } (h i, h j) \in \mbox{int}(\Omega_D)_h \setminus (\Gamma_D)_h \cap (\Omega_D)_h,\\
(v_{i,j})_2= \frac{g_{i,j+1}- g_{i,j}}{h}, &\mbox{ for all } (h i, h j) \in  (\Gamma_D)_h\cap (\Omega_D)_h,\\
(v_{i,j})_2=0, & \mbox{ for all } (h i, h j) \in (\partial(\Omega_D)_h \cap \Omega_h) \setminus (\Gamma_D)_h.\\
\end{array} \right .
\end{equation}
\begin{figure}[htp]
\begin{center}
\includegraphics[width=2.05in]{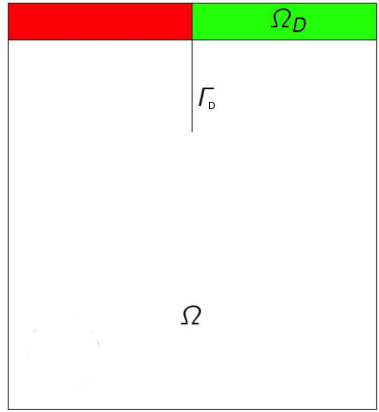}
\end{center}
\caption{We show an example of domain $\Omega$ and its sub-portion $\Omega_D$ as well as the interface $\Gamma_D$ as mentioned in the description
in the Francfort-Marigo model..}\label{domainsfr}
\end{figure}

We depict the typical situation described by this model in Figure \ref{domainsfr}, where the different colors represent the different uniform values of $g$ on the domain $\Omega_D$ and $\Gamma_D$ the interface of discontinuity. For simplicity, let us summarize all these linear conditions into a linear constraint
$$
A v =f,
$$
for a suitable matrix $A$. 
Then, we can rewrite the functional $\mathcal J_{h}$, with a slight abuse of notation, as
\begin{equation}\label{FM2d}
\mathcal J_{h}(v):= h^2 \sum_{i,j} W^2_{\sqrt{\beta/h}} (v_{i,j}) 
.
\end{equation}
to be minimized with respect to $v$ under the constraint $A v = f$.  Let us stress now that $g$ is actually a boundary datum, and provided the derivative field $v$ of $u$ we can recover $u$ simply by line integration.
For a suitable coordinate system on the discrete domain, we define accordingly the line integration operator 
\begin{equation}\label{lineint}
u_{i,j} =  \sum_{m=0}^j v_{i,m} + g_{0,m}, \quad \mbox{for all } (hi,hj) \in \Omega_h,
\end{equation}
which can be expressed in the compact form by 
\begin{equation}\label{lineint2}
u = B v + g.
\end{equation}
Notice that, while for the Mumford-Shah
functional we needed to consider the psuedo-inverse matrix $D^\dagger$ of the discrete differentiation operator $D$
within the fidelity term $\|T v - \tilde g\|$, in order to have the orthogonality relationships \eqref{orthog} and \eqref{orthog2}, here the $D^\dagger$ is of no use. We also mention that
 the minimization of \eqref{FM2d} is again of the general type \eqref{specfunc}
for the choice of the maps $U_k =  h^2 W^2_{\sqrt{\frac{\beta}{h}}}$ for all $k=(i,j)$. The issue of the coercivity of $\mathcal J_{h}$ on the affine space $\{A v = f\}$ is addressed in the work \cite{arcafosoXX} on discrete rate-independent evolutions.
\\

The ability of predicting complicated crack paths is the greatest strength of the Francfort-Marigo
model and the reason for its popularity. Nevertheless Francfort and Marigo acknowledged in their seminal
work \cite{frma98} that, from a mechanical point of view, it would be preferable to define an evolution by
means of local minimizers. We shall show in Section \ref{numexper}  an application 
of our Algorithm \eqref{ncbregman} where we actually perform a simulation of a fracture in a one 
dimensional model, evolving through critical points, being a numerically robust  and  physically sound description
of the happening of the fracture.

\subsection{Quasi-static evolution of cohesive fractures}\label{cohesive}

In these models one considers the fracture growth in an elastic body when the cohesive forces acting between the lips of the crack are not negligible. Again we address the case
in which the evolution is driven by a time-dependent boundary displacement on a fixed portion of
the boundary \cite{ca08,cato11}. For the sake of simplicity we provide only a one dimensional description
of the model, mentioning that there are not principle difficulties to extend what follows to higher dimension.

Let $A, R > 0$, and consider the following energy functional:
\begin{equation}\label{coesfunc}
\mathcal J (u) = \int_{0}^A |u'|^2 \, dx + \int_A^{2A} |u'|^2 \, dx + g_{\Gamma} (|u^+ - u^-|),
\end{equation}
where $u: (0,2A) \to \mathbb{R}$ is a $SBV$-function, $\Gamma$ a fixed point in $(0,2A)$,   and $g: [0, \infty) \to [0, \infty)$ is given by
$$
g (s) = 
\begin{cases}
 \vspace{.1cm}
 - \frac{s^2}{2 R} + s &\text{ if } 0 \leq s < R, \\
\frac{R}{2} &\text{ if } 0 \leq s \geq R.
\end{cases}
$$
As described in \cite{ca08,cato11}, and similarly to the quasi-static evolution proposed in the  brittle fracture
model by Francfort and Marigo in Section \ref{brittle}, the evolution of the body configuration develops along
critical points of the energy \eqref{coesfunc}. Hence it is crucial to be able to compute  critical points of $\mathcal J$ under the boundary conditions $u (0) = f(0) = a$, $u(2A) = f (2A) = b$, 
with $f: (0, 2A) \to \mathbb{R}$ given.

Again, we discretize the space variable by setting $h=1/N$ for $N \in \mathbb N$ and
$$
u_i : = u (i h A) , \qquad i = 0, \ldots, 2 N.
$$
Let us assume now that $\Gamma$ is placed exactly in the middle of the domain, i.e., $\Gamma =A$.
Then, the functional $\mathcal J$ can be approximated by its discretized version
\begin{eqnarray*}
\mathcal J_h(u) &=& \sum_{i=0}^{N-1} A h \left| \frac{u_{i+1} - u_i}{A h}  \right|^2 +  \sum_{i=N+1}^{2N-1} A h \left| \frac{u_{i+1} - u_i}{A h}  \right|^2 
+ g_{\Gamma} (|u_{N+1} - u_N|) \\
&=& \frac{1}{A h}  \sum_{i=0}^{N-1} | u_{i+1} - u_i |^2 + \frac{1}{A h} \sum_{i=N+1}^{2N-1} | u_{i+1} - u_i |^2 
+ g_{\Gamma} (|u_{N+1} - u_N|).
\end{eqnarray*}
Setting $v_i:= u_{i+1} - u_i$ for every $i = 0, \ldots, 2N - 1$, we may rewrite the previous expression as
\begin{equation}\label{FC1d1}
\mathcal J_h (v) = \frac{1}{A h}  \sum_{i=0}^{N-1} | v_i |^2 + \frac{1}{Ah} \sum_{i=N+1}^{2N-1} | v_i |^2 
+ g_{\Gamma} (| v_N|).
\end{equation}

Finally we seek for critical points of $\mathcal J_h(v)$ subject to 
\begin{equation}\label{FC1d2}
v \in \mathcal A(f) = \left\{ z \in  \mathbb{R}^{2N}:  \sum_{i=0}^{2N -1} z_i = f (2A) - f (0) \right\}.
\end{equation}
It is not difficult to see now that this latter problem is precisely of the type \eqref{specfunc} for 
the choice of $U_k =0$ for all $k \neq N$, $U_N(\cdot) = g_\Gamma(|\cdot|)$, and $T=I-P_N$, $P_N$ being the projection on the $N$-th coordinate.
Let us now stress that the functional $\mathcal J_h$ is actually coercive over the set of feasible
competitors and also $\omega$-semi-convex. Indeed, for any choice of $v^0 \in \mathbb R^{2N}$ the functional 
$$
\mathcal J_h (v) + \omega A h \sum_{i=0}^{2N -1} |v_i - v_{i}^0|^2
$$
is strictly convex as soon as
$$
\omega > \frac{N}{2A} \| g '' \|_{L^{\infty}((0,\infty))} = \frac{1}{2A h R}.
$$
Hence, our Algorithm \eqref{ncbregman} is directly applicable. We refer to \cite{arcafosoXX} for more details
about discrete rate independent evolutions driven by Algorithm \eqref{ncbregman}.
\\

We shall conclude this section by mentioning that there are many more models, even beyond continuum mechanics, where Algorithm \eqref{ncbregman} is 
very robustly applicable with great impact. We mention, for instance, that the solution of minimization problems of functionals of the type \eqref{FM2d}, by means 
of Algorithm \eqref{ncbregman}, where the constraint $A$ matrix is a compressed sensing matrix, has been used in \cite{arfopeXX} to 
greatly outperform $\ell_1$-minimization as a decoding procedure in case of noise on the signal prior to the
measurement via $A$, reducing significantly the so-called {\it noise-folding phenomenon}. 
\\

With the scope of clarifying in detail the applicability of Algorithm \eqref{ncbregman} and because of the relevance of the truncated
quadratic potential $W_r^2$ in so many applications (image processing, quasi-static evolutions of brittle fractures, compressed sensing,
etc.), we focus on the application of Algorithm \eqref{ncbregman} to problems of the type \eqref{specfunc}
for $U_k = W_r^p$, where $W_r^p(t) = \min \{|t|^p, r^p\}$, for $r>0$, $p\geq 1$, and any $t \in \mathbb R$.
In particular, we shall discuss in detail  the role of the coercivity of the objective functionals as well as
 how the main conditions (A1) and (A2)  can be verified in practice.

\subsection{Truncated polynomial minimization}\label{truncated}

First of all, we should mention that, independently of the choice of the linear operators $T$ and $A$, by \cite[Theorem 2.3]{fowa10}, the constrained minimization problem
\begin{equation} \label{minprob}
\hbox{Minimize } \J_p(v) = \| T v - g \|^2 + \gamma \sum_{i=1}^m W_r^p(v_i), \hbox{ Subject to } A v = f,
\end{equation} 
has always global minimizers. Notice that the proof of existence of minimizers is far from being trivial (see Remark \ref{exrem} below), since the problem is in general not coercive. 
Concerning uniqueness and stability of minimizers, we refer instead to the work of Durand and Nikolova \cite{duni06, duni06-1}, about cases where $T$ is injective on $\ker A$.
\begin{remark}\label{exrem}
The proof of existence of solutions of \eqref{minprob} is based on a special orthogonal decomposition of certain convex sets, see \cite[Appendix, Section 8.1]{fowa10}. Let us report the main fact, which it will turn out to be useful to us again later in this paper.
\\
Define $\bar \J_p(v) = \| T v - g \|^2 + \gamma \sum_{i=1}^m c_i |v_i|^p$ for $c_1, \dots c_m$ scalars; notice that we allow some of them to be negative or zero, as soon as $\bar \J_p(v) \geq C_{\inf} > - \infty$ for all $v \in \EE$.
 Then for any constant $C >0$ and any polyhedral convex set $X \subset \EE$,  there exists a linear subspace $\mathcal V= \mathcal V_{X,C} \subset  \EE$, such that the orthogonal projection $X^\perp$ of $X$ onto $\mathcal V^\perp$ has the properties 
\begin{itemize}
\item $X = \{ x =  x^\perp \oplus t v: x^\perp \in X^\perp, v \in \mathcal V, t \in \mathbb R^+\}$,
\item $M_C = X^\perp  \cap \{ v \in \EE: \bar \J_p(v) \leq C\}$ is compact, and 
\item $\bar \J_p (\xi_t)$ is constant along rays $\xi_t=x^\perp \oplus t v$, where $x^\perp \in M_C$, $v \in  \mathcal V$, and $ t \in \mathbb R^+$. 
\end{itemize}
For $\mathcal I_0 \subset \mathcal I$ and $\mathcal U_{\mathcal I_0} :=\{ v \in \EE: |v_i| \leq r, i \in \mathcal I_0 \mbox{ and }  |v_i|> r, i \in \mathcal I \setminus \mathcal I_0 \}$,
in particular this result applies on $X=  \mathcal F(f)\cap \overline{\mathcal U_{\mathcal I_0}}$, hence
\begin{equation} \label{minprob2}
\hbox{Minimize } \bar \J_p(v) = \| T v - g \|^2 + \gamma \sum_{i=1}^m c_i |t|^p, \hbox{ Subject to } A v = f \mbox{ and } v \in \mathcal U_{\mathcal I_0},
\end{equation} 
has solutions in $\EE$, actually in the compact set $M_{\bar \J_p(v^0)} = X^\perp  \cap \{ v \in \EE: \bar \J_p(v) \leq \bar \J_p(v^0)\}$, for any $v^0 \in \EE$.
\end{remark}

\subsection{Issues about the applicability of the algorithm}\label{issuesapp}

Due to the nonsmoothness and nonconvexity of $\J_p$, the more general linearly  constrained minimization \eqref{minprob} has been so far an open problem, as standard methods, such as SQP and Newton methods, do not apply, unless one provides a $C^2$-regularization of the problem. In particular, it would be desirable that an appropriate algorithm performing 
such an optimization could retain both the simplicity of the thresholding iteration and its unconditional convergence properties, as given by \cite[Theorem 4.8]{fowa10} in the case of unconstrained minimization of the functional $\mathcal J_p$. Certainly the method \eqref{ncbregman} is a strong candidate,
as the iterations of its inner loop actually requires only a unconstrained minimization, which can be again addressed by iterative thresholding, see Section \ref{innerthr} below. However, we encounter two major bottlenecks to the direct application of this algorithm to \eqref{minprob}. The first problem is that $\J_p$ does not satisfy our main assumption (A1), i.e., it is not $\omega$-semi-convex, as it is not a $C^1$-perturbation of a convex functional. In fact the term $W_r^p$ is too rough at the kink where the truncation applies. The second trouble
comes by the lack of coerciveness of  $\J_p$ on the affine space $\F( f)$ in general, for a generic choice of $T$. A general convergence result (Theorem \ref{Msmainthm}) will be therefore available only under an additional condition on $T$.

\subsection{A smoothing method}\label{smoothapp}

In this section we would like to construct an appropriate {\it slightly smoother} perturbation $\J_p^\varepsilon$ of $\J_p$, which allows eventually for $\omega$-semi-convexity, but does not modify essentially the minimizers over $\F(f)$. 
Such modification will not affect the possibility of using thresholding functions in the numerical setting, although instead of the hard-type discontinuous thresholding encountered in the unconstrained case, as in  \cite[Proposition 4.3]{fowa10} and \cite{fowa10}[Figure 2], our new thresholding function will be a Lipschitz one, as an effect of the introduced regularization, in dependence of the choice of the parameters $\gamma, \varepsilon, r, \omega$ in appropriate ranges. We will see in Section \ref{innerthr} the usefulness of this feature in terms of guaranteed exponential convergence from the beginning of the iterations.

We start by the following polynomial interpolation result.
\begin{lemma}\label{lem2}
Let $0<s_1<s_2$ and assume that
$$
\pi(t) := A(t-s_2)^3 +B(t-s_2)^2 + C,
$$
is a third degree polynomial. Given $\gamma_1,\gamma_2,\gamma_3 \in \mathbb R$ and by setting
\begin{equation}\label{Bdefin}
\left \{ 
\begin{array}{l}
C=\gamma_3, \\
B=\frac{\gamma_1}{s_2-s_1} - \frac{3 (\gamma_3 - \gamma_2)}{(s_2 -s_1)^2},\\
A=\frac{\gamma_1}{3(s_2-s_1)^2} + \frac{2 B}{3(s_2 -s_1)},
\end{array}
\right .
\end{equation}
then we have the following interpolation properties
\begin{equation}
\left \{ 
\begin{array}{ll}
\pi(s_2) = \gamma_3, & \pi(s_1) = \gamma_2, \\
\pi'(s_2) = 0, & \pi'(s_1) = \gamma_1.
\end{array}
\right .
\end{equation}
\end{lemma}
\begin{proof}
The equalities related to $s_2$ are straightforward, the others related to $s_1$ follow by simple direct computations:
\begin{eqnarray*}
\pi(s_1) &=& -\frac{\gamma_1}{3}(s_2-s_1) - \frac{2}{3} B(s_2-s_1)^2 + B(s_2-s_1)^2 + \gamma_3\\
&=& -\frac{\gamma_1}{3}(s_2-s_1) +  \frac{B}{3} (s_2-s_1)^2 + \gamma_3 \\
&=& -\frac{\gamma_1}{3}(s_2-s_1) + \frac{\gamma_1}{3}(s_2-s_1) - (\gamma_3 - \gamma_2) + \gamma_3 = \gamma_2,
\end{eqnarray*}
and
\begin{eqnarray*}
\pi'(s_1) &=& 3 A (s_1-s_2)^2 +2 B(s_1-s_2)\\
&=& \gamma_1 - 2 B(s_1-s_2) + 2B (s_1 - s_2 ) = \gamma_1.
\end{eqnarray*}
\end{proof}

Given $0<\varepsilon<r$ for every $t \in [r-\varepsilon, r+\varepsilon]$ we define $\pi_p(t) =\pi(t)$ as in Lemma \ref{lem2} for $s_1 = (r-\varepsilon)$, $s_2=(r+\varepsilon)$, $\gamma_1 = p(r-\varepsilon)^{p-1}$,
$\gamma_2 = (r-\varepsilon)^p$, and $\gamma_3 = r^p$. For example, for $p=2$, we have
$$
\pi_2(t) = \frac{[t+(r-\varepsilon)][\varepsilon(r+t)-(r-t)^2]}{4 \varepsilon}, \quad t \in \mathbb R.
$$

\begin{figure}[htp]
\begin{center} 
\includegraphics[width=4.2in]{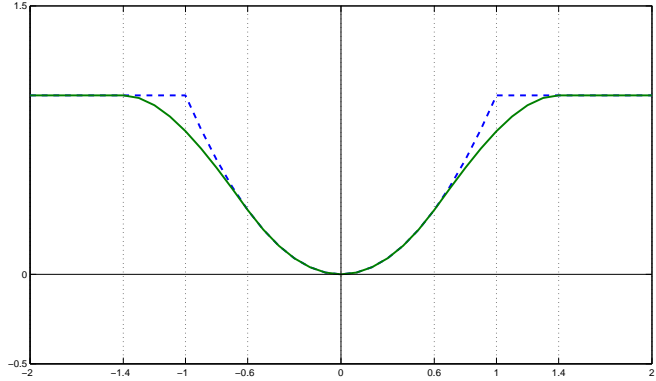}
\end{center}
\caption{Truncated quadratic potential $W_r^p$ and its regularization $W_r^{p,\varepsilon}$, for $p=2$, $r=1$, and $\varepsilon=0.4$.}\label{trunc}
\end{figure}

Let us now set,  for all $t \geq 0$,
\begin{equation}\label{epsW}
W_r^{p,\varepsilon}(t) = 
\left \{
\begin{array}{ll}
t^p, & t \leq r-\varepsilon,\\
\pi_p(t), & r-\varepsilon \leq t \leq r+\varepsilon,\\
r^p & t \geq r+\varepsilon,
\end{array}
\right .
\end{equation}
whereas for $t \leq 0$, we define $W_r^{p,\varepsilon}(t) = W_r^{p,\varepsilon} (-t)$. Notice that now $ W_r^{p,\varepsilon}$ is actually a $C^1$-function of $\mathbb R$, for all $0<\varepsilon <r$. This smoothing is reminiscent of the approach to graduated nonconvexity proposed in the seminal work \cite{BZ}.

Thanks to the function $W_r^{p,\varepsilon}$, we can define the following perturbation of $\J_p$ 
\begin{equation} \label{pfunceps}
\J_p^\varepsilon(v) = \| T v - g \|^2 + \gamma \sum_{i=1}^m W_r^{p,\varepsilon}(v_i).
\end{equation}
About the existence of constrained minimizers of $\J_p^\varepsilon$ we have the following abstract result, whose proof is again shifted to the Appendix. We stress that the result holds for all inverse free-discontinuity problems, as it requires no further assumptions on $T$.
\begin{theorem}\label{existthm}
For $0 \leq \varepsilon < r $, the problem
\begin{equation} \label{minprob3}
\hbox{Minimize } \J_p^\varepsilon(v) = \| T v - g \|^2 + \gamma \sum_{i=1}^m W_r^{p,\varepsilon}(v_i), \hbox{ Subject to } A v = f,
\end{equation} 
has solutions in $\EE$. Actually, such minimal solutions can be taken in a compact set $M \subset \EE$ independent of $\varepsilon$.
\end{theorem}


\begin{remark}\label{quasiequicoerc}
The previous result clarifies that, despite the fact that in general $\J_p^\varepsilon$ are not coercive functionals, up to restricting them 
to an appropriate compact set, independent of $\varepsilon$, they can be considered equi-coercive.
\end{remark}

\begin{corollary}\label{gammaconv}
The net of functionals $( \J_p^\varepsilon)_{0 \leq \varepsilon <r}$ $\Gamma$-converges to $\J_p$ on $\mathcal F(f)$. Moreover, if 
we consider the net of minimizers $v^*_\varepsilon$ of $\J_p^\varepsilon$ in $M$ for $0 \leq \varepsilon <r$, as constructed in Theorem \ref{existthm} (which are actually minimizers of  $\J_p^\varepsilon$ over $\mathcal F(f)$ as well), then the accumulation points of such a net are minimizers of $ \J_p$.
\end{corollary}
\begin{proof}
As $\J_p^\varepsilon$ converges uniformly to $\J_p$ on $\mathcal F(f)$, we deduce immediately its  $\Gamma$-convergence \cite{DM}. By Theorem \ref{existthm} and compactness of $M$ we conclude the convergence of minimizers.
\end{proof}

\begin{proposition}\label{condverified}
For all $0<\varepsilon<r$, the functional $\J_p^\varepsilon$ satisfies the properties (A1) and (A2), i.e., it is $\omega$-semi-convex, and \eqref{strana} holds.
\end{proposition}
\begin{proof}
The $\omega$-semi-convexity follows from the piecewise continuity and boundedness of the second derivatives of $\J_p^\varepsilon$. Since $W_r^{p,\varepsilon}(t)\ge 0$ and $|(W_r^{p,\varepsilon})'(t) | \leq p r^{p-1}$ for every $t \in \R$, by means of the elementary inequality $a\le \tfrac12 (a^2+1)$ we obtain
\begin{eqnarray}
 \| \nabla \J_p^\varepsilon(v)\| &\leq& 2 \| T^*(T v - g)\| + \gamma \| ((W_r^{p,\varepsilon})'(v_1), \dots, (W_r^{p,\varepsilon})'(v_m))\| \nonumber\\
&\leq & 2 \| T^*\| \| T v - g \| + \gamma   m^{1/2} p r^{p-1} \label{bndTv}\\
&\leq & \| T^*\| \| T v - g \|^2 + \| T^*\| + \gamma   m^{1/2} p r^{p-1} \nonumber\\
&\leq & \| T^*\| \J_p^\varepsilon(v) + \| T^*\| + \gamma   m^{1/2} p r^{p-1} \nonumber\,.
\end{eqnarray}
Hence, for $K=\| T^*\|$ and $L=\| T^*\| + \gamma   m^{1/2} p r^{p-1}$, we get that \eqref{strana} holds
for $\J_p^\varepsilon$. 
\end{proof}

\subsection{The application of the algorithm to coercive cases}\label{Msh}

As we clarified in the previous section, functionals of the type $\mathcal J_p^\varepsilon$, for $0 < \varepsilon <r$, satisfy the assumptions (A1) and (A2) for the applicability
of the algorithm \eqref{ncbregman}. In particular, when the algorithm is applied for $\mathcal J = \mathcal J_p^\varepsilon$, then by Theorem \ref{chiave} the sequence  $(v_\ell)_{\ell \in \mathbb N}$ generated by the algorithm has the 
properties
\begin{itemize}
 \item[(a)] $(Av_\ell-f) \to 0$ as $\ell \to \infty$;
 \item[(b)] $(v_\ell-v_{\ell-1}) \to 0$ as $\ell \to \infty$.
\end{itemize}
However $\mathcal J_p^\varepsilon$ is unfortunately not necessarily coercive on $\mathcal F(f) =\{v \in \EE: A v = f\}$, 
although it retains some coerciveness by considering suitable compact subsets $M$ of competitors, see Theorem \ref{existthm}. 
Nevertheless, such information does not help when it comes to the application of the algorithm \eqref{ncbregman}, as there is no natural or simple way of restricting or projecting the iterations
to such compact sets $M$. 
Hence, in order to apply Theorem \ref{mainthm}, we need to explore the mechanism for which the iterations $(v_\ell)_{\ell \in \mathbb N}$ generated by the algorithm
keep bounded. We show that this is the case where $T$ is injective on $\ker A$, since this allows us to recover the coerciveness we need.

Let us first introduce some specific notation for the application of the algorithm \eqref{ncbregman}, in particular we denote
\begin{equation}\label{omconv}
\mathcal J_{\omega,u}(v) := \mathcal J_{p,\omega,u}^\varepsilon (v)= \mathcal J_p^\varepsilon(v) + \omega \|v-u\|^2.
\end{equation}

\begin{lemma}\label{gllip}
For all $0 < \varepsilon < r$,  the sequence $(\| \nabla \J_p^\varepsilon(v_\ell)\|)_{\ell \in \mathbb N}$ is uniformly bounded, where the iterations $(v_\ell)_{\ell \in \mathbb N}$ are generated by the algorithm \eqref{ncbregman}.
\end{lemma}
\begin{proof}
As a consequence of \eqref{456} the sequence $(\|T \bar v_\ell \|)_\ell$, where $\bar v_\ell$ is defined in \eqref{barv}, is uniformly bounded. From \eqref{17853}, we have also that $(\|T v_\ell \|)_\ell$ is uniformly bounded. 
As pointed out in \eqref{bndTv} of Proposition \ref{condverified} actually we have $\| \nabla \J_p^\varepsilon(v_\ell)\| \leq 2 \| T^*\| \| T v_\ell - g \| + \gamma   m^{1/2} p r^{p-1}$. Hence  the sequence $(\| \nabla \J_p^\varepsilon(v_\ell)\|)_{\ell \in \mathbb N}$ is uniformly bounded. 
\end{proof}

The next lemma will be crucial to show the convergence of the algorithm in our case.

\begin{lemma}\label{unbfAq} 
For all $0<\varepsilon<r$,  the sequence $(A^* q_{\ell,L_\ell-1})_{\ell \in \mathbb N}$ generated by the application
of the algorithm \eqref{ncbregman} for $\mathcal J = \mathcal J_p^\varepsilon$ is uniformly bounded.
\end{lemma}
\begin{proof}
By \eqref{inclu} we have
$$
A^* q_\ell \in \nabla \mathcal J_{\omega,v_{\ell-1}} (v_\ell) = \nabla \mathcal J_{p}^\varepsilon (v_\ell) + 2 \omega (v_\ell - v_{\ell-1})
$$
As, by Lemma \ref{gllip}, $\nabla \mathcal J_{p}^\varepsilon (v_\ell)$ is uniformly bounded and $v_\ell - v_{\ell-1} \to 0$, for $\ell \to \infty$, we obtain that also $A^* q_\ell$ is
uniformly bounded. By \eqref{ncbregman}, we have also
$$
A^* q_\ell = A^* q_{\ell,L_\ell-1} - 2 \lambda A^* (A v_\ell - f),
$$
from which, together with $(A v_\ell - f) \to 0$ for $\ell \to \infty$, we eventually deduce the uniform boundedness of $ A^* q_{\ell,L_\ell-1} $ as well.
\end{proof}

\begin{lemma}\label{bndvl}
Assume that $T$ is injective on $\ker A$, or $\ker T \cap \ker A = \{0\}$. Then, for all $0<\varepsilon< r$,  the sequences $(v_\ell)_\ell$  generated by the application
of the algorithm \eqref{ncbregman} for $\mathcal J = \mathcal J_p^\varepsilon$ is uniformly bounded. 
\end{lemma}
\begin{proof}
Notice that, by definition of $v_\ell$ in \eqref{ncbregman}, necessarily it solves the following linear system 
$$
 (T^*T +\frac{1}{2} A^* A) v_\ell = \frac{1}{2}A^* (f +q_{\ell, L_{\ell}-1}) + \omega (v_{\ell-1} - v_\ell),
$$
where the right-hand-side of this equality if uniformly bounded by Lemma \ref{unbfAq} and Theorem \ref{chiave} (b). Moreover, as $(Av_\ell-f) \to 0$ for $\ell \to \infty$, we can write that $v_\ell$ is solution of the system
$$
\underbrace{\left [
\begin{array}{l}
 (T^*T +\frac{1}{2} A^* A)\\
A
\end{array}
\right]}_{:=G}
 v_\ell = w_\ell,
$$
where the right-hand-side $w_\ell$ is actually uniformly bounded with respect to $\ell$. Due to our assumption $\ker T \cap \ker A = \{0\}$, we obtain that $\ker G=\{0\}$ and 
$$
v_\ell = (G^* G)^{-1} G^* w_\ell,  \mbox{ for all } \ell \in \mathbb N,$$
hence the uniform boundedness of $(v_\ell)_\ell$.
\end{proof}

We summarize this list of technical observations into the following convergence result.

\begin{theorem}\label{Msmainthm}
Assume that $T$ is injective on $\ker A$, or $\ker T \cap \ker A = \{0\}$. Then, for all $0<\varepsilon<r$,  the sequences $(v_\ell)_\ell$  generated by the application
of the algorithm \eqref{ncbregman} for $\mathcal J = \mathcal J_p^\varepsilon$ has at least one accumulation point, and every accumulation point is a constrained critical point of
$\mathcal J_p^\varepsilon$ on the affine space $\mathcal F(f) = \{v \in \EE: A v = f\}$.
\end{theorem}
\begin{proof}
The result follows by a direct application of Theorem \ref{mainthm}, after having recalled the boundedness of  $(v_\ell)_\ell$, which results from Lemma \ref{bndvl}.
\end{proof}

\begin{remark}
The previous convergence result actually applies for the case of the Mumford-Shah functional, for which $T = D_h^\dagger$ and $A = I - D_h D_h^\dagger$, since $D_h^\dagger$ is in fact injective on $\ran(D_h)$, see Section \ref{MSfunc}.
\end{remark}

\subsection{Iterative thresholding algorithms revisited}\label{innerthr}

As already mentioned an iterative thresholding algorithm can be used for identifying local minimizers of the $\mathcal J_p$, see \cite{fowa10} for details. This algorithm is actually very attractive for its exceptional simplicity, and
its ability of performing a {\it separation of components} at a finite number of iterations, leading eventually to a contractive iteration and its convergence.

In this section, we would like to show how an iterative thresholding algorithm can play a profitable role also for linearly constrained problems of the type \eqref{minprob}: namely, it allows for the construction of a very simple and efficient procedure for solving the inner loop minimization problems in our algorithm. Differently from the unconstrained case, however, the thresholding function we can use is a continuous one, so that we do not need to prove a result of separation of components after a finite number of iterations, and we gain additionally contractivity, unconditionally and from the beginning of the iteration.

For the sake of simplicity and without loss of generality, we consider the application of \eqref{ncbregman} for $\lambda =1/2$, and we define now the $\nu$-strongly convex functional 
\begin{equation}\label{omconv2}
\mathcal J_{\omega,u}(v,q) := \mathcal J_{p,\omega,u}^\varepsilon (v,q)= \mathcal J_{\omega,u}(v) + \frac{1}{2} \| Av - (f +q)\|^2.
\end{equation}
Requiring $\nu$-strong convexity is equivalent to the following lower bound on $\omega$.

\begin{lemma}\label{omega}
Define $\mathcal J_{\omega,u}(v,q)$ as in \eqref{omconv2}. According to the notation introduced in Lemma \ref{lem2}, let $B$ as in \eqref{Bdefin} for $s_1 = (r-\varepsilon)$, $s_2=(r+\varepsilon)$, $\gamma_1 = p(r-\varepsilon)^{p-1}$,
$\gamma_2 = (r-\varepsilon)^p$, and $\gamma_3 = r^p$. Then, for
\begin{equation}\label{omlow}
 \omega > \gamma |B| = \gamma \left | \frac{p(r-\varepsilon)^{p-1}}{2 \varepsilon} +  \frac{3}{4 \varepsilon^2} [(r-\varepsilon)^p - r^p] \right|,
\end{equation}
$\mathcal J_{\omega,u}(v,q)$ is a $\nu$-strongly convex function of $v$.
\end{lemma}

\begin{proof}
It obviously suffices to show that $\mathcal J_{\omega,u} (v)$ is $\nu$-strongly convex, and since
$$
\mathcal J_{\omega,u} (v)=\| T v - g \|^2 + \gamma \sum_{i=1}^m W_r^{p,\varepsilon}(v_i) + \omega \sum_{i=1}^m (v_i-u_i)^2
$$
it is enough to check that for every $s \in \R$ the real function 
$$t \to \gamma W_r^{p,\varepsilon}(t) + \omega \sum_{i=1}^m (t-s)^2$$
is $\nu$-strongly convex. But this function is piecewise $C^2$ with bounded second derivatives, thus we must only check that for every $t$ such that $|t|\notin\{r-\varepsilon, r+\varepsilon\}$ (notice that the latter is a set and not an interval!), there exists $\nu>0$ such that $\gamma (W_r^{p,\varepsilon})''(t) + 2\omega \ge\nu >0$. By the explicit expression \eqref{epsW} of $W_r^{p,\varepsilon}(t)$, it all reduces to check that for every $t \in (r-\varepsilon, r+\varepsilon)$ one has
\begin{equation}\label{goal}
\gamma\pi''_p(t) + 2\omega \ge\nu >0\,.
\end{equation}

Now, in our case we have
\begin{eqnarray}
B&=& \frac{p(r-\varepsilon)^{p-1}}{2 \varepsilon} +  \frac{3}{4 \varepsilon^2} [(r-\varepsilon)^p - r^p],\label{Bconst}\\
A&=& \frac{p(r-\varepsilon)^{p-1}}{12 \varepsilon^2} +  \frac{B}{3 \varepsilon}. \label{Aconst}
\end{eqnarray}
Since $(r-\varepsilon)^p - r^p \leq - \varepsilon p(r-\varepsilon)^{p-1}$ by convexity, we deduce from \eqref{Bconst} that
\begin{equation}
B \leq - \frac{p(r-\varepsilon)^{p-1}}{4 \varepsilon}<0,
\end{equation}
and therefore, from \eqref{Aconst} we have also $A \leq 0$. But then for all $t \in (r-\varepsilon, r+ \varepsilon)$, we deduce from
these negativity relationships and again \eqref{Aconst} that
\begin{equation*}
\pi_p''(t) = 6 A ( t- (r+\varepsilon)) +2 B \geq 2B
\end{equation*}
so that \eqref{omlow} implies \eqref{goal}, with $\nu=2(\omega-\gamma|B|)$, as required.
\end{proof}

We now define our thresholding function. We preliminarily fix some notation. We fix $0<\varepsilon<r$, and for $W_r^{p,\varepsilon}(t)$ as in \eqref{epsW}, $B$ as in \eqref{Bconst}, and a positive parameter $\mu$ such that
\begin{equation}\label{voraussetzung}
\mu|B|< 1,\,
\end{equation}
we consider
\begin{equation} \label{funceps}
f_{\mu,r}^\xi(t) = (t - \xi)^2 + \mu W_r^{p,\varepsilon}(t)
\end{equation}
with $\xi$ a real number. By \eqref{voraussetzung}, arguing as in the proof of Lemma \ref{omega}, we get $\nu$-strong convexity of $f_{\mu,r}^\xi(t)$, therefore we can define a function $S^{\mu}_p(\xi)$ through
\begin{equation}\label{threshs}
S_p^{\mu}(\xi)=\arg \min_{s \in \R} f_{\mu,r}^\xi(s)\,. 
\end{equation}
Then $S^{\mu}_p(\xi)$ satisfies the following properties.
\begin{lemma}\label{eigenschaft}
For every $\xi \in \R$ and $\mu$ as in \eqref{voraussetzung}, the function $S^{\mu}_p(\xi)$ satisfies:
\begin{itemize}
 \item[(a)]$S_p^{\mu}(\xi)=t$ if and only if $2(t - \xi) + \mu (W_r^{p,\varepsilon})'(t)=0$.
 \item[(b)]$S_p^{\mu}(\xi)$ is a strictly increasing function.
 \item[(c)]$S_p^{\mu}(\xi)$ is Lipschitz continuous with 
\begin{equation}\label{lip}
{\rm Lip}(S_p^{\mu})\le \frac{1}{1-\mu|B|}\,.
\end{equation}
\end{itemize}
\end{lemma} 

\begin{proof}
Part (a) of the statement is obvious by \eqref{funceps}, \eqref{threshs} and the $\nu$-strong convexity of $f_{\mu,r}^\xi(t)$. To prove part (b), fix $\xi_1 < \xi_2 \in \R$ and correspondingly, let $t_1:=S^{\mu}_p(\xi_1)$ and $t_2:=S^{\mu}_p(\xi_2)$. We have $2(t_1 - \xi_1) + \mu (W_r^{p,\varepsilon})'(t_1)=0$. Assume by contradiction that $t_2 \le t_1$. Since the function $t \mapsto 2(t - \xi_1) + \mu (W_r^{p,\varepsilon})'(t)$ is strictly increasing by strong convexity, we get $2(t_2 - \xi_1) + \mu (W_r^{p,\varepsilon})'(t_2)\le 0$. Now $\xi_1 < \xi_2$ yields $2(t_2 - \xi_2) + \mu (W_r^{p,\varepsilon})'(t_2)< 0$, in contradiction with part (a) of the statement.

To prove part (c), we fix $\xi_1$ and $\xi_2 \in \R$, and we can suppose without loss of generality that $\xi_1 <\xi_2$. Again, we define $t_1:=S^{\mu}_p(\xi_1)$, and $t_2:=S^{\mu}_p(\xi_2)$. From part (b), we have $t_1 < t_2$ and from part (a) we get that
$$
2(t_1 - \xi_1) + \mu (W_r^{p,\varepsilon})'(t_1)=2(t_2 - \xi_2) + \mu (W_r^{p,\varepsilon})'(t_2),\,
$$
that is, since $\xi_1 < \xi_2$ and $t_1 <t_2$, 
\begin{equation}\label{conto}
|t_2-t_1|+\frac \mu2 [(W_r^{p,\varepsilon})'(t_2)-(W_r^{p,\varepsilon})'(t_1)]=|\xi_2-\xi_1|\,.
\end{equation}
Now $(W_r^{p,\varepsilon})'(t)$ is piecewise $C^1$ with bounded derivative. Moreover, given $B$ as in \eqref{Bconst}, arguing as in Lemma \ref{omega} we have $(W_r^{p,\varepsilon})''(t)\ge 2B$ for every $t$ such that $|t|\in (r-\varepsilon, r+\varepsilon)$. Since $(W_r^{p,\varepsilon})''(t)\ge 0$ when $|t|\notin [r-\varepsilon, r+\varepsilon]$ and $B<0$, we get that $(W_r^{p,\varepsilon})''(t)\ge 2B$ for every $t$, with the only exceptions of the four points $t=r-\varepsilon$, $t=-r-\varepsilon$, $t=r+\varepsilon$, and $t=-r+\varepsilon$. Since $t_1 <t_2$, by the fundamental theorem of calculus we have
$$
[(W_r^{p,\varepsilon})'(t_2)-(W_r^{p,\varepsilon})'(t_1)]\ge 2B(t_2-t_1)=-2|B||t_2-t_1|\,,
$$
since $B <0$. Using \eqref{conto}, this gives
$$
(1-\mu|B|)|t_2-t_1|\le|\xi_2-\xi_1|,\,
$$
which concludes the proof.
\end{proof}
\begin{figure}[htp]
\begin{center}
\subfigure[]{ \label{s1(a)}
\includegraphics[width=2.4in]{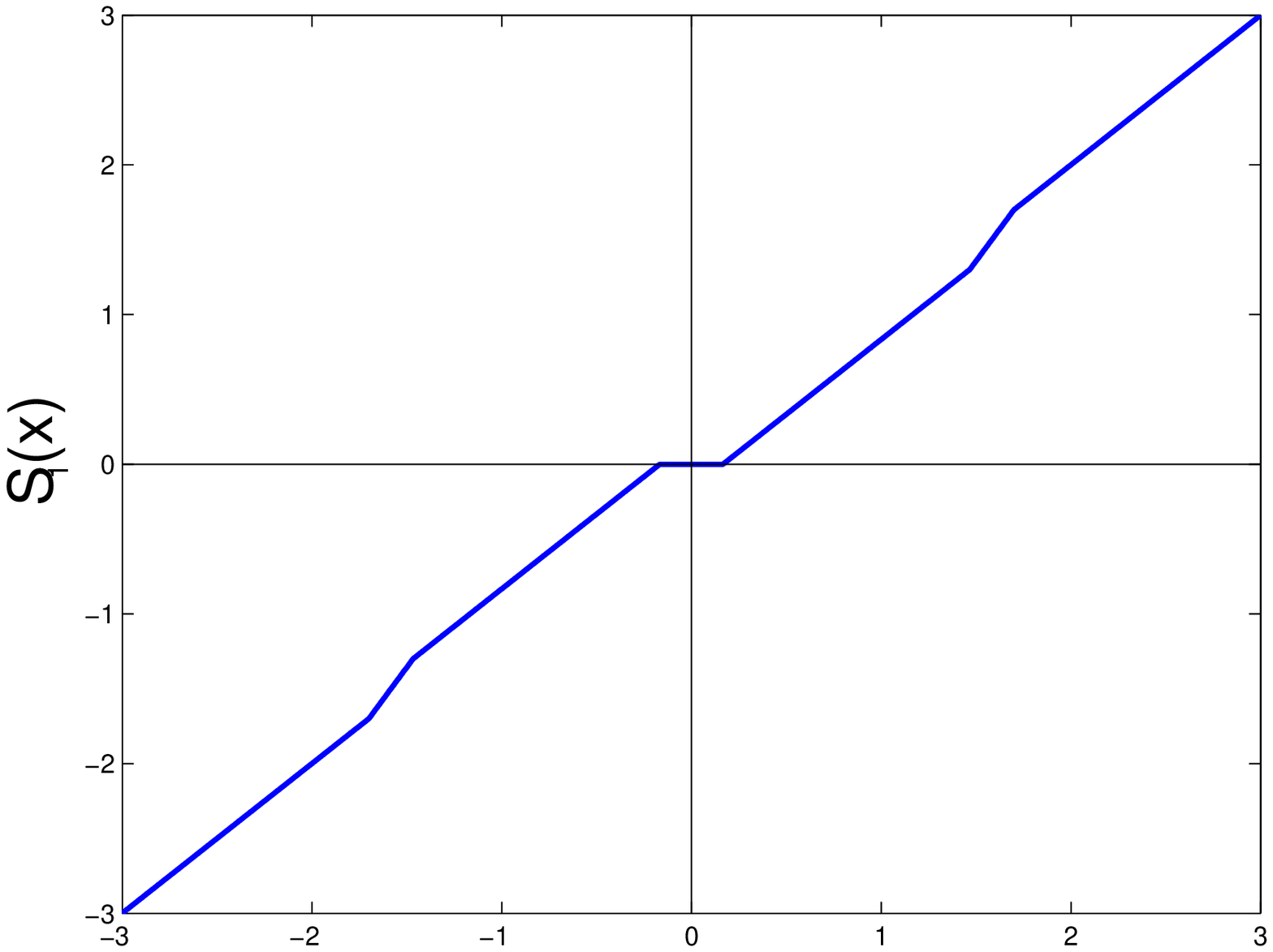}}
\subfigure[]{\label{s1(b)}
\includegraphics[width=2.4in]{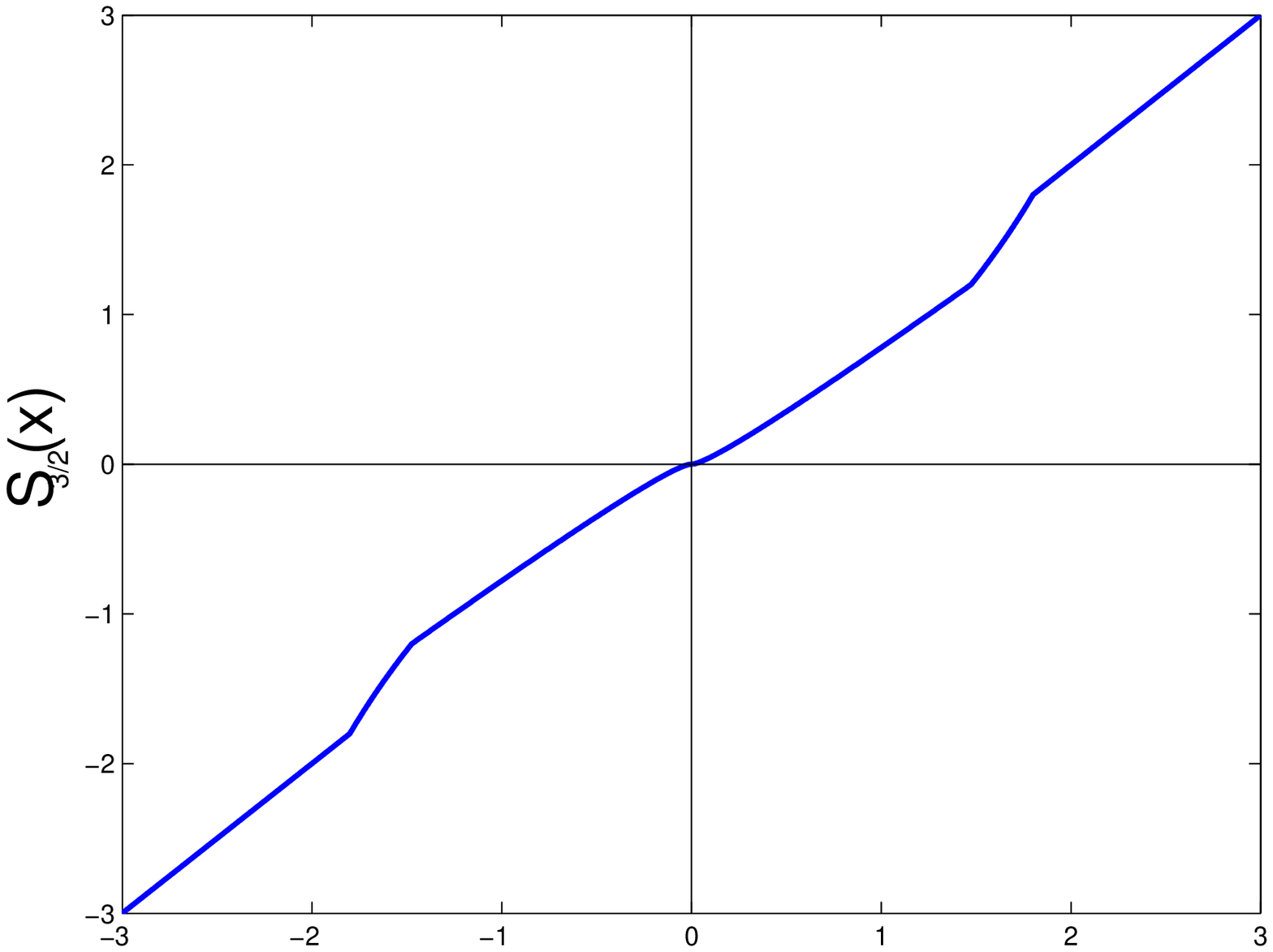}}
\subfigure[]{\label{s1(c)}
\includegraphics[width=2.4in]{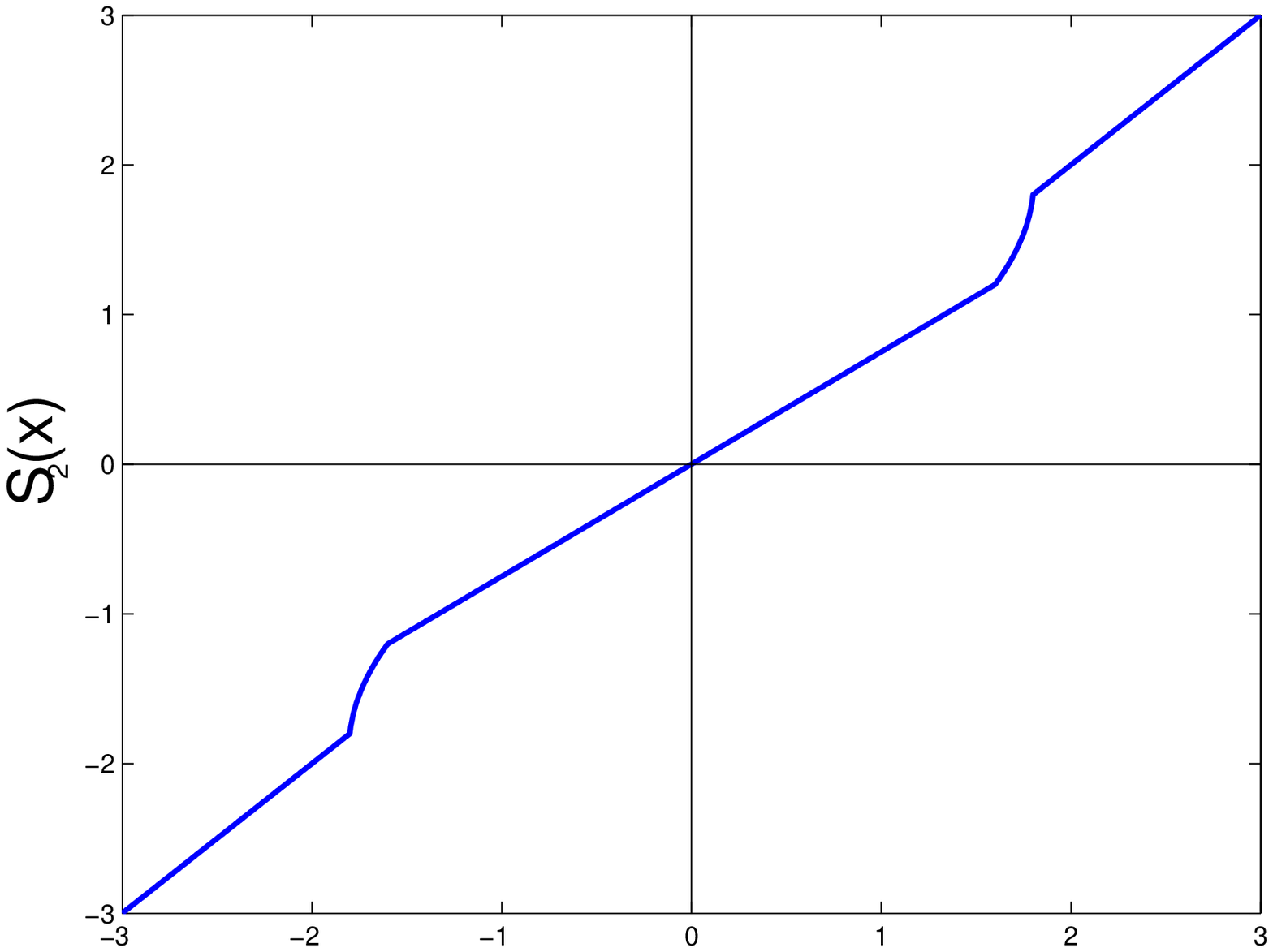}}
\end{center}
\caption{The Lipschitz continuous thresholding functions  $S_1^\mu$,  $S_{3/2}^\mu$, and $S_{2}^\mu$, with parameters $p=1,3/2$, and $2$, respectively, and $r = 1.5$, $\mu = 5$, $\varepsilon=0.3$.}\label{Sthreshold}
\end{figure}

While this latter result states certain qualitative properties of $S_p^\mu$ for any $1 \leq p < \infty$, we explicitly write its expression, e.g., for $p=2$, which is easily obtained by solving a second degree polynomial equation:
\begin{equation}\label{s2thrs}
S_2^\mu(\xi)=\left\{\begin{array}{lr}
\displaystyle\frac{\xi}{1+\mu} & |\xi|<(r-\varepsilon)(1+\mu)\\ \ \\
\displaystyle\frac{4\varepsilon}{3\mu}\left(1+\frac{\mu}{4\varepsilon}(2\varepsilon+r)-\sqrt{\frac{\Gamma(\xi)}{4}}\right)&(r-\varepsilon)(1+\mu)\leq|\xi|\leq r+\varepsilon\\ \ \\
\displaystyle\xi&|\xi|>r+\varepsilon
\end{array}\right.
\end{equation}
where
\begin{equation}
\Gamma(\xi)= 4 \left ( 1+\left(\frac{\mu}{4\varepsilon}\right)^2(2r+\varepsilon)^2+\frac{\mu}{2\varepsilon}(r+2\varepsilon)-\frac{3\mu}{2\varepsilon}\xi \right).
\end{equation}
We further report in Figure \ref{Sthreshold} the graphics of the thresholding function $S_p^\mu$ for $p \in \{1,3/2,2\}$, and parameters $r = 1.5$, $\mu = 5$, $\varepsilon=0.3$.

We now get back to our functional $\mathcal J_{\omega,u}(v,q)$ defined in \eqref{omconv2} and we further consider the associated surrogate functional, 
\begin{eqnarray}
\mathcal J_{\omega,u}^{surr}(v,q,w) := \mathcal J_{p,\omega,u}^{surr,\varepsilon} (v,q,w)= \mathcal J_{\omega,u}(v,q) &+& (\| v-w\|^2 - \| T v- Tw\|^2) \nonumber \\
&+&(\| v-w\|^2 - \frac12 \| A v- Aw\|^2) \nonumber \\
&+& (\| v-w\|^2 - \omega \| v- w\|^2).\label{surromconv2}
\end{eqnarray}
Up to rescaling of $g, f, q, \gamma$ we can assume here and later, and  without loss of generality, that $\| T\| <1$, $\frac{1}{\sqrt 2} \| A\| <1$, and $\omega < 1$, while still keeping the lower bound on $\omega$ given by \eqref{omlow} which is necessary to ensure $\nu$-strong convexity. Hence, we have
\begin{equation}\label{mon1}
\mathcal J_{\omega,u}^{surr}(v,q,w) \geq \mathcal J_{\omega,u}(v,q),
\end{equation}
and
\begin{equation}\label{mon2}
\mathcal J_{\omega,u}^{surr}(v,q,w) = \mathcal J_{\omega,u}(v,q),
\end{equation}
if and only if $w = v$.

\begin{proposition} 
Let $0<\varepsilon < r$. Assume  $\| T\| <1$, $\frac{1}{\sqrt 2} \| A\| <1$, $\omega < 1$, and that $\omega$ and $\gamma$ satisfy \eqref{omlow}. 
Then
\begin{equation}\label{mincond1}
v^* = \arg \min_{v \in \EE} \mathcal J_{\omega,u}(v,q)
\end{equation}
if and only if $v^*$ satisfies the following component-wise fixed-point equation: for $i=1,\dots,m$,
\begin{equation}\label{fixpt1}
v_i^* = S_p \left ( \frac{1}{3} \left \{ [ (I-T^*T) + (I-\frac12A^*A) + (1-\omega)I] v^* +  (T^* g + \frac{1}{2} A^*(f+q) +  \omega u ) \right \}_i \right), 
\end{equation}
where $S_p$ is the thresholding function $S_p^{\mu}$ defined in Lemma \ref{eigenschaft} for $\mu=\gamma/3$. 
\end{proposition}

\begin{proof}
Assume that $v^*$ satisfies \eqref{mincond1}. From \eqref{mon1} and \eqref{mon2}, we have the inequalities
\begin{eqnarray*}
\mathcal J_{\omega,u}^{surr}(v^*,q,v^*) &=& \mathcal J_{\omega,u}(v^*,q)\\
&\leq & \mathcal J_{\omega,u}(v,q) \\
&= & \mathcal J_{\omega,u}^{surr}(v,q,v)\\
&\leq & \mathcal J_{\omega,u}^{surr}(v,q,v^*).
\end{eqnarray*}
Hence we obtain also
\begin{equation}\label{uzti}
v^* = \arg \min_{v \in \EE} \mathcal J_{\omega,u}^{surr}(v,q,v^*). 
\end{equation}
We notice now by a direct computation that
\begin{equation}\label{ertqw}
\frac 13 \mathcal J_{\omega,u}^{surr}(v,q,v^*) = \left \| v - \left ( \frac{b^1+b^2+b^3}{3}  \right) \right \|^2 + \frac{\gamma}{3} \sum_{i=1}^m W_r^{p,\varepsilon}(v_i) + C(b^1,b^2,b^3,\gamma),
\end{equation}
where $b^1 = (I-T^*T) v^* + T^*g$, $b^2 = (I-\frac{1}{2}A^*A) v^* + \frac{1}{2}A^*(f+q)$, $b^3=  (I-\omega I) v^* + \omega u$, and $C(b^1,b^2,b^3,\gamma)$ is a term which does not depend on $v$.
It now follows by the definition \eqref{threshs} of $S_p=S_p^{\mu}$ that $v^*$ satisfies \eqref{fixpt1}.

Conversely, by \eqref{ertqw}, if $v^*$ satisfies \eqref{fixpt1}, then it also satisfies \eqref{uzti}. It follows that
$$
0 \in \partial \mathcal J_{\omega,u}^{surr}(v^*,q,v^*)=\partial \mathcal J_{\omega,u}(v^*,q)
$$
where the last equality trivially follows from \eqref{surromconv2}. By convexity of $\mathcal J_{\omega,u}$, this implies \eqref{mincond1}.
\end{proof}

Looking at the fixed point equation \eqref{fixpt1}, which characterizes the unique minimizer of $\mathcal J_{\omega,u}(v,q)$, it is now natural to wonder whether the corresponding fixed-point iteration
\begin{equation}\label{fixpt3}
v_i^{n+1} = S_p \left ( \frac{1}{3} \left \{ [ (I-T^*T) + (I-\frac12A^*A) + (1-\omega)I] v^n +  (T^* g + \frac{1}{2} A^*(f+q) +  \omega u ) \right \}_i \right),
\end{equation}
generates a sequence $(v^n)_{n \in \mathbb N}$ which converges to $v^*$. The next Theorem gives a positive answer to this question.

\begin{theorem}\label{thmthrconv}
Let $0<\varepsilon < r$. Assume  $\| T\| <1$, $\frac{1}{\sqrt 2} \| A\| <1$, $\omega < 1$, and that $\omega$ and $\gamma$ satisfy \eqref{omlow}. 
Let $v^*=\arg \min_{v \in \EE} \mathcal J_{\omega,u}(v,q)$, and consider the sequence $v^n$ defined by the iteration \eqref{fixpt3}. Let $\delta:=\frac{3-\omega}{3-\gamma|B|}$, where $B$ is defined by \eqref{Bconst}. Then $\frac 23 <\delta<1$ and for every $n \in \N$ one has
\begin{equation}\label{konvergenz}
\|v^{n}-v^*\|\leq \frac{\delta^n}{1-\delta} \|v^1-v^0\|
\end{equation}
so that in particular $v^n\to v^*$ as $n$ tends to $+\infty$.
\end{theorem}

\begin{proof}
By the assumptions $\omega<1$ and \eqref{omlow}, the bounds on $\delta$ are obvious. For every $n \geq 0$ one has $v^{n+1}= \mathbb U(v^n)$, where $\mathbb U$ is an operator having component-wise action defined by
$$
[\mathbb U(v)]_i= S_p\left ( \frac{1}{3} \left \{ [ (I-T^*T) + (I-\frac12A^*A) + (1-\omega)I] v +  (T^* g + \frac{1}{2} A^*(f+q) +  \omega u ) \right \}_i \right)\,,
$$
where $S_p$ is the function $S_p^{\mu}$ defined in Lemma \ref{eigenschaft} for $\mu=\gamma/3$.
Using the hypotheses, it is easy to show that $\|\frac{1}{3}[(I-T^*T) + (I-\frac12A^*A) + (1-\omega)I]\|\le 1-\tfrac \omega3$, therefore, using \eqref{lip} for $\mu=\gamma/3$ we get
$$
\operatorname {Lip}(\mathbb U)\le (1-\tfrac \omega3)(\tfrac{1}{1-\gamma/3 |B|})=\delta\,;
$$
in particular, $\mathbb U$ is a contraction mapping, and we conclude by Banach fixed point Theorem.
\end{proof}

\section{Numerical Experiments}\label{numexper}

In this section we report the results of numerical experiments to demonstrate and confirm the behavior of the algorithm as predicted by our theoretical findings. 
We focus on two relevant examples, i.e., the minimization of the discrete Mumford-Shah functional in dimension two, and the discrete time quasi-static evolution of the 
Francfort-Marigo brittle fracture model in one dimension.  

\subsection{Mumford-Shah functional minimization in dimension two}

As clarified in \eqref{MS2d}, this is equivalent to consider the minimization of the function $\J_p(v)=\| T v - g \|^2 + \gamma \sum_{i=1}^m W_r^p(v_i)$, for $p=2$, subjected to $A v = 0$, where $T =D_h^\dagger$, $A = I - D_h D_h^\dagger$,
and $v$ represents the two dimensional discrete gradient of the competitor $u$. In all the simulations, we used the iterative thresholding algorithm \eqref{fixpt3} in order to solve the convex optimizations of the inner loop.
Our first experiment refers to the implementation of the algorithm for competitors $u$, being two dimensional arrays, of dimensions $25 \times 25$. 
The parameters chosen are $\gamma=1.7 \times 10^{-1}$, $r=3.5$, and $\varepsilon=4.5\times 10^{-3}$. Notice that $\omega$ is always explicitely fixed according to the formula $\omega >\gamma \left ( \frac{1}{4} + \frac{r}{2\varepsilon} \right)$, as one can easily derive by combining \eqref{omlow} and \eqref{Bconst}, for $p=2$. 
In Figure \ref{im25} we show the dynamics of the discrepancy $\| A v_\ell \|$ to the realization of the linear constraint $A v = 0$, and of the energy
$\J_p(v_\ell)$, depending on the iterations $v_\ell$, for $\ell=0,1,2,\dots$. This simulation confirms that the algorithm tends to converge to a stationary point with energy level lower than the initial guess, and for which the constraint is numerically verified.

\begin{figure}[htp]
\begin{center}
\subfigure[]{ \label{A25(a)}
\includegraphics[width=2.4in]{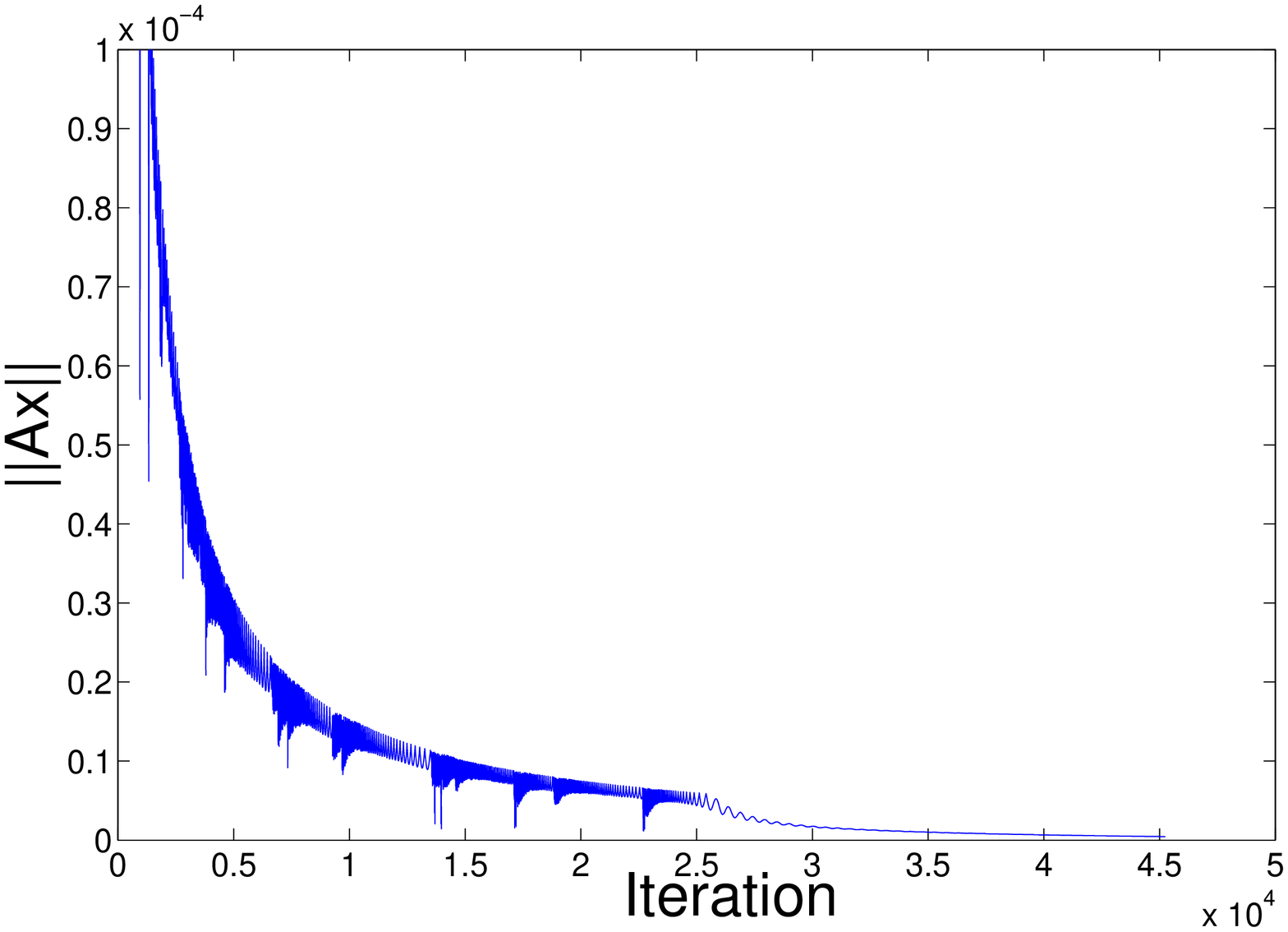}}
\subfigure[]{\label{J25(b)}
\includegraphics[width=2.4in]{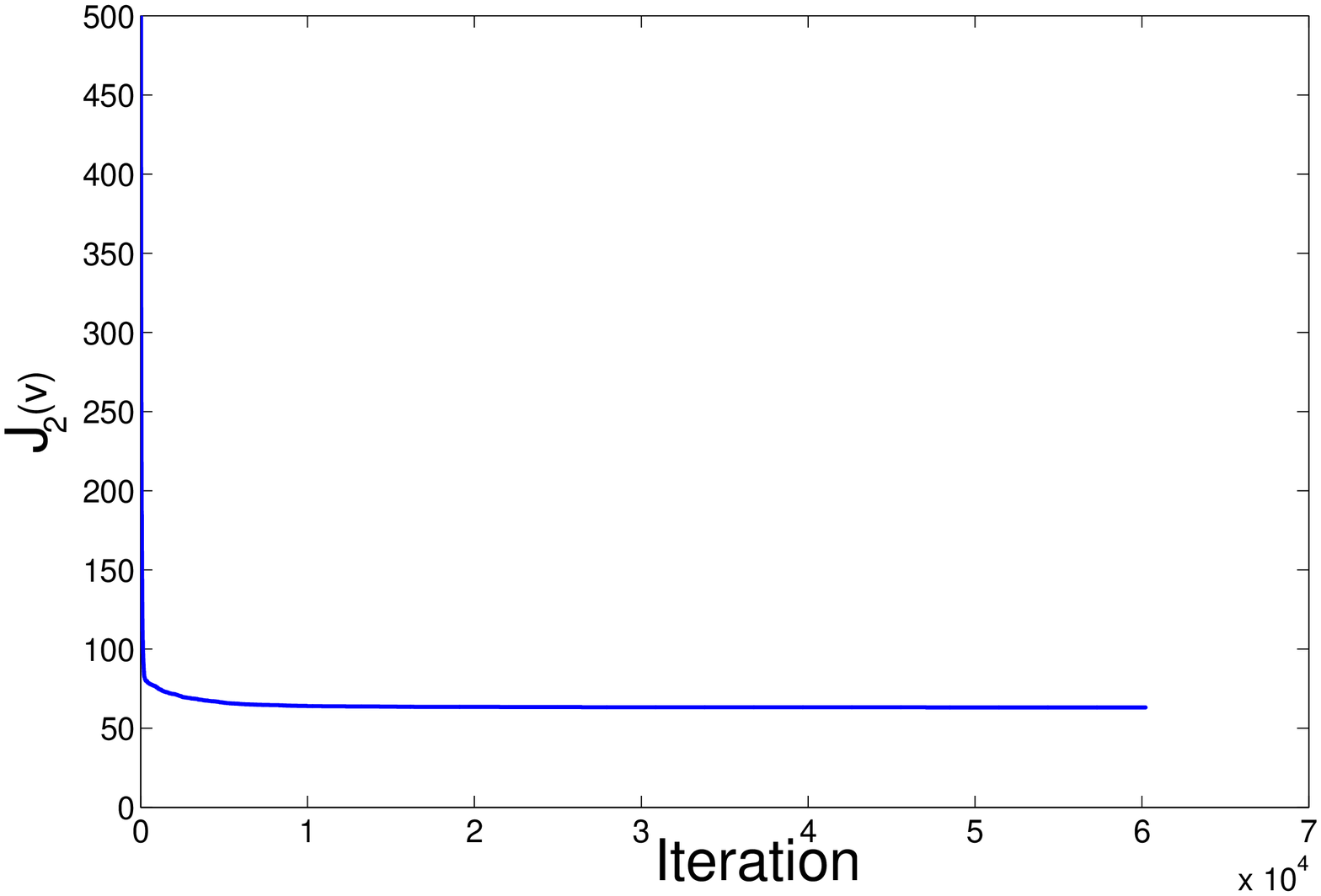}}
\end{center}
\caption{On the left subfigure we show the dynamics of the discrepancy $\| A v_\ell \|$ to the realization of the linear constraint $A u = 0$, and on the right subfigure
the one of the energy $\J_p(v_\ell)$, depending on the iterations $v_\ell$, for $\ell=0,1,2,\dots$}\label{im25}
\end{figure}

For a qualitative evaluation of the behavior of the algorithm, we report below an experiment on a denoising problem for an image of dimensions $125 \times 125$, see Figure \ref{papageno}, where the original image, the noisy version, 
and its denoised version after minimization are reported respectively in the subfigures (a), (b), and (c). The numerical experiments is conducted with $6 \%$ noise, and parameters $\gamma=1.4\times 10^{-1}$, $r=2.8$, and $\varepsilon=3.5 \times 10^{-3}$.

 \begin{remark}\label{remfin}
As mentioned at beginning of this section, our numerical experiments are exclusively aimed at verifying the setting of the parameters and the convergence of the Algorithm \eqref{ncbregman},
with no claim of optimal implementation. However, for the sake of completeness, we mention here how to treat the most demanding numerical issues.
As the algorithm requires the applications of the matrices $D_h^\dagger$ and $(D_h^\dagger)^*$, one may wonder whether such matrices can be efficiently, stably computed and applied. 
In principle, when enough memory is available there is no problem in computing such matrices in advance, also symbolically, and obtaining an iteration at machine precision. If the available memory
is limited, one may avoid to  attempt the explicit computation of such (pseudo)inverses, as it is also a good practice in numerical analysis. Rather one should use a preconditioned
iterative method to approximate the results of their applications, i.e., $D_h^\dagger v$ and $(D_h^\dagger)^*u$. For any matrix $X$ the following identities hold:
\begin{eqnarray*}
 X^* X X^\dagger &=& X^*, \\
X X^* (X^\dagger)^* &=& X.
\end{eqnarray*} 
In case of $X= D_h$ the first identity gives us a method to compute $D_h^\dagger v$, as it is consequently sufficient to solve the linear system
\begin{equation}
\label{divgrad}
\left \{
\begin{array}{ll}  (D_h^* D_h) u = D_h^* v, \\ c({u}) =0,
\end{array}
\right. 
\end{equation}
and set $D_h^\dagger v = u$, where $c(u)=c_{u}$ is the mean value of $u$. 
In order to see that the discrete system \eqref{divgrad} can be efficiently and stably solved we need to highlight its relationships with a corresponding
continuous system. In fact  \eqref{divgrad} actually can be simply interpreted as the discretization of the following continuous partial differential equation, which we write in its weak form
$$
\left \{
\begin{array}{ll}
  \int_\Om \nabla u \cdot \nabla \varphi = \int_\Om v \cdot \nabla \varphi & \\
\int_\Omega u(x) dx = 0&
\end{array}
\right .
$$
for all $\varphi \in H^1(\Om)$. Such an elliptic partial differential equations can be  approached numerically very stably and efficiently by FEM or finite difference discretizations and solved by suitable preconditioned iterations, for instance by means of multigrid methods \cite{hack85}. 
Similarly one can approach the computation of $(D_h^\dagger)^*u$ by defining the system 
\begin{equation}
\label{graddiv}
\left \{
\begin{array}{ll}  (D_h D_h^*) v =  D_h u, \\ A v =0,
\end{array}
\right. 
\end{equation}
and setting $(D_h^\dagger)^* u = v$, where $A$ is again the discrete $\operatorname{curl}$ operator (notice that this matrix is sparse!). The efficient solution of the system \eqref{graddiv} is again subordinated to the use of suitable preconditioners. 
\\
Concerning the overall computational cost one may wonder whether the solution of two systems of (discretized) PDE such as \eqref{divgrad} and \eqref{graddiv} is indeed an exceedingly large amount of effort. To this issue, let us respond that other well-known and established methods for the minimization of Mumford-Shah functional require also the solution of elliptic PDEs, for instance the Ambrosio and Tortorelli approach \cite{amto90}. 
\\
In general, one  can still object that the  operator $T^* T = (D_h^\dagger)^* D_h^\dagger$ as it appears in the iteration \eqref{fixpt3} is likely to be ill conditioned and this might affect negatively the
convergence. However, as it is shown in Theorem \ref{thmthrconv}, as soon as  the operators $T$ and $A$ are properly rescaled and the parameters $\omega$ and $\gamma$ are suitably set, the inner-loop \eqref{fixpt3} is guaranteed
to converge with exponential rate. How the interplay of the spectral properties of the operators $T$ and $A$ may affect the  convergence rate of the outer loop of Algorithm \eqref{ncbregman} is instead likely to be a very difficult problem
to be analyzed, which is general enough to be beyond the scope of this paper.
\end{remark}

\begin{figure}[htp]
\begin{center}
\subfigure[]{ \label{pappa(a)}
\includegraphics[width=2.3in]{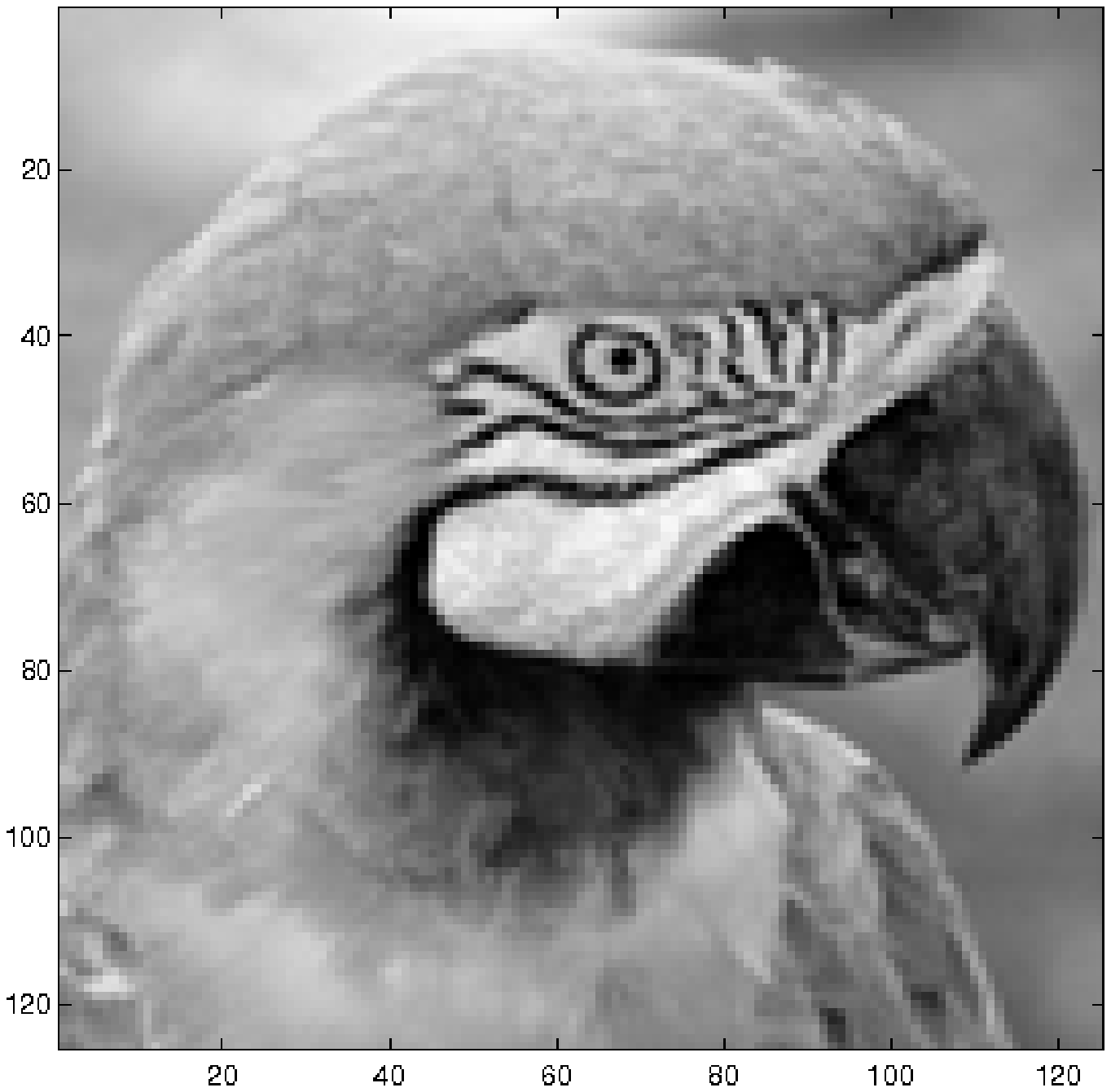}}
\subfigure[]{\label{pappanoise(b)}
\includegraphics[width=2.3in]{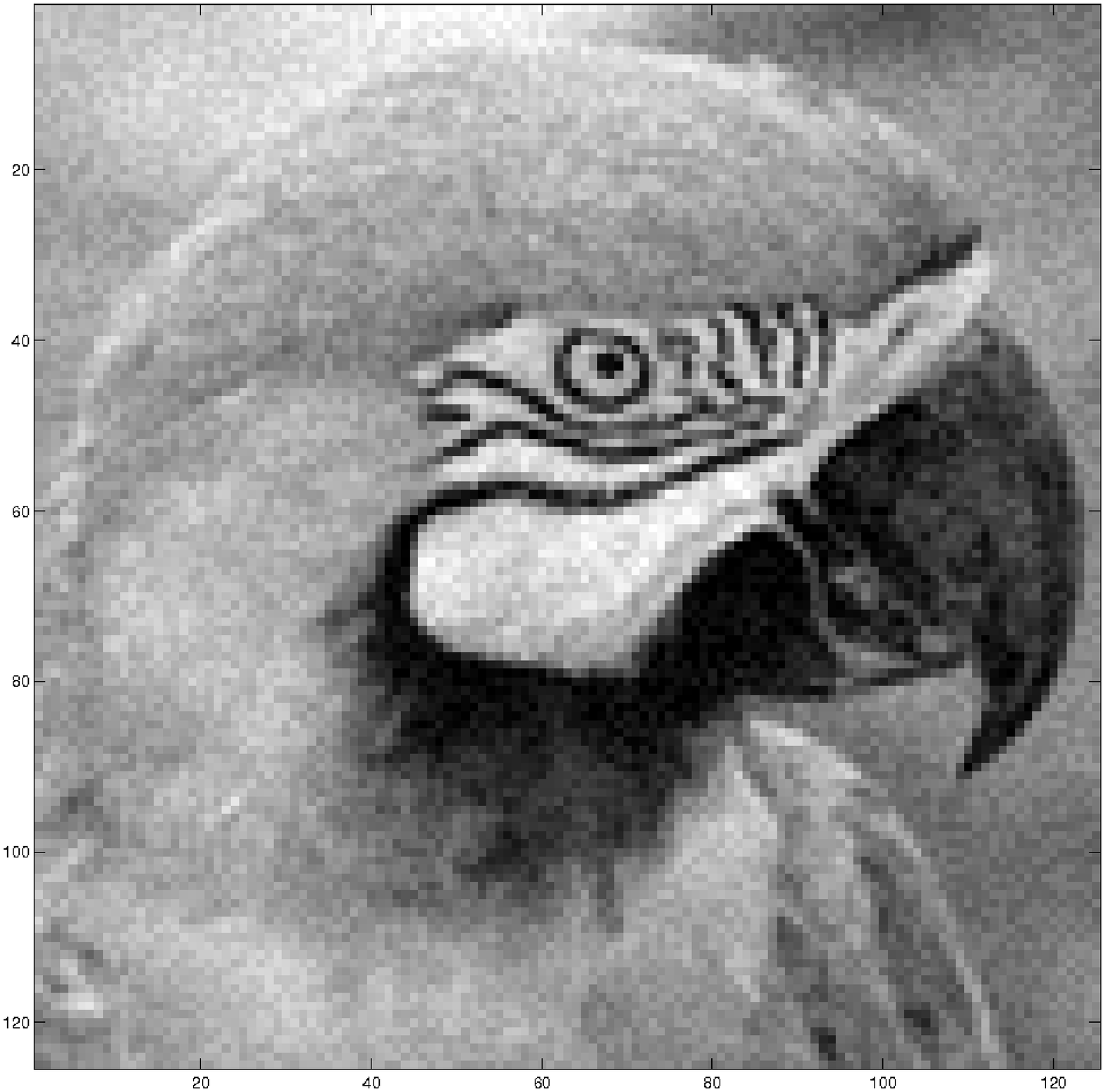}}
\subfigure[]{\label{papparec(c)}
\includegraphics[width=2.3in]{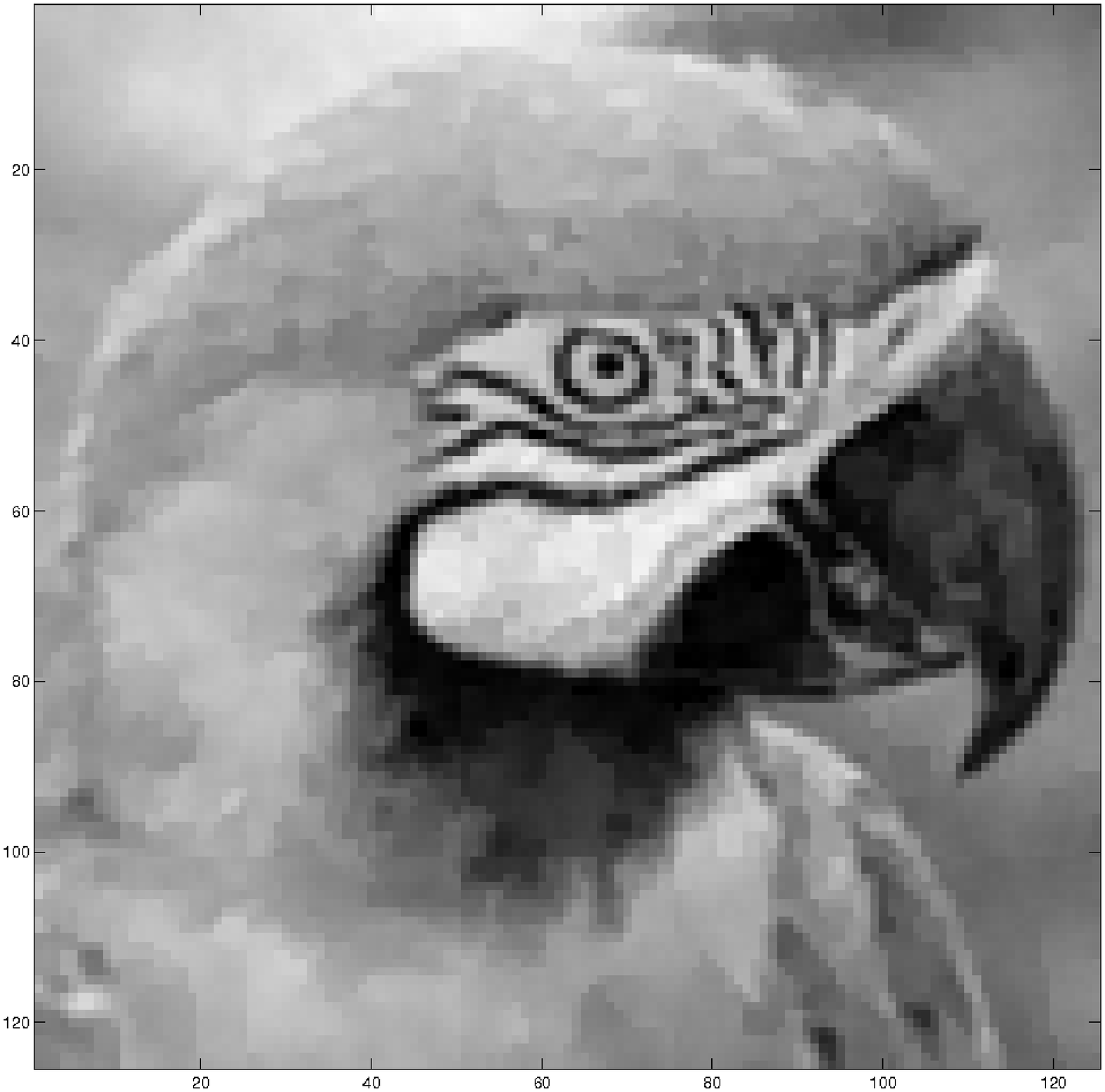}}
\end{center}
\caption{Application of algorithm \eqref{ncbregman} with inner loop realized by iterative thresholding \eqref{fixpt3} for a classical denoising problem.}\label{papageno}
\end{figure}

\subsection{Brittle fracture simulation}

In this subsection we  show the result of the discrete time evolution of the Francfort-Marigo model of brittle fracture in one dimension
as presented in Section \ref{brittle}. Here we assume that $\Omega$ is the interval $[0,1]$ and that the load is applied 
on the boundary $\Omega_D=\{0,1\}$ of the interval. The load corresponds to a displacement $g(t,0)=-t$ and $g(t,1)=t$ at the boundary where $t\geq 0$ is
the time variable. For the minimization of the functional \eqref{FM2d} we again use Algorithm \eqref{ncbregman} with parameters $\gamma = 1$,
$\varepsilon = 10^{-3}$, $r = 2$, and $\omega = \frac{1}{2} \left (\frac{1}{2} + \frac{r}{(N-1)\varepsilon} \right )$, where $N=51$ is the number of space discretization points.
The evolution proceeds with time steps of width $\Delta t = 0.01$. At every new time step, the initial guess for the application of Algorithm \eqref{ncbregman} is the 
state of the gradient of the displacement at the previous time step.
In Figure \ref{francfortmarigosim} (a) we show four stages of the displacement solution $u$ at the time $t=0,\ 0.4,\ 0.8,\ 1.45$. As one can notice the rod
is deformed initially in an elastic way, until the crack happens at multiple positions at time $t=0.9$, being a more favorable critical point of the energy.
In  Figure \ref{francfortmarigosim} (b) we report the evolution of the energy \eqref{discrete2} in time,  where the rupture time is highlighted also by the elastic energy collapse.
\\
As clarified in Section \ref{brittle}, let us again stress that for this model there is no need of computing the action of the pseudoinverse matrix $D_h^\dagger$.  The simulation of the entire evolution until the crack takes few minutes (a few seconds per time iteration) on a standard personal computer with a Matlab implementation. In Figure \ref{tempo} we show the computational time required at each discrete time and we observe how the algorithm needs to 
search longer for the new critical point, as soon as the physical phase transition from elastic evolution to fracture happens.
The numerical results are consistent with the predicted analytical solutions to this well-known model \cite{arfopeXX,frma98}, showing the robustness of Algorithm \eqref{ncbregman}
towards the simulation of physical models.

\begin{figure}[htp]
\begin{center}
\subfigure[]{ \label{displ(a)}
\includegraphics[width=3.5in]{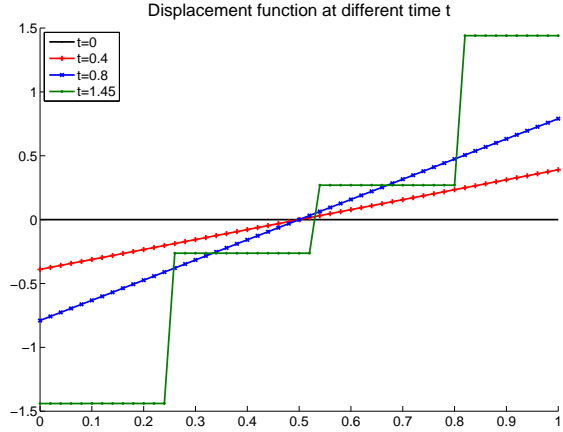}}
\subfigure[]{\label{energy(b)}
\includegraphics[width=3.5in]{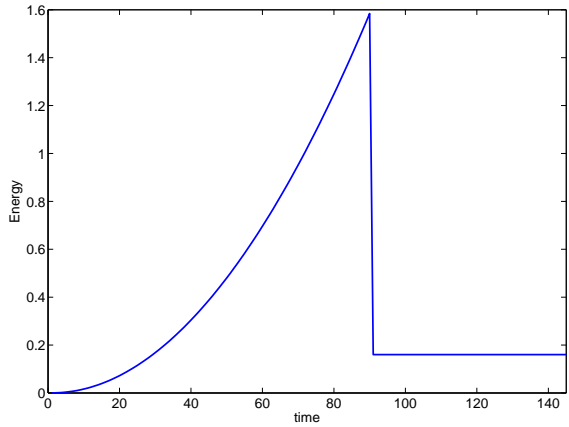}}
\end{center}
\caption{Discrete time evolution of the Francfort-Marigo model of brittle fracture. In the subfigure (a) we show four stages of the rod displacement at different times, starting
with an elastic evolution until crack formation. In the subfigure (b) we show the evolution of the energy \eqref{discrete2}, where the rupture time is highlighted also by the elastic energy collapse.}\label{francfortmarigosim}
\end{figure}

\begin{figure}[htp]
\begin{center}
\includegraphics[width=3.5in]{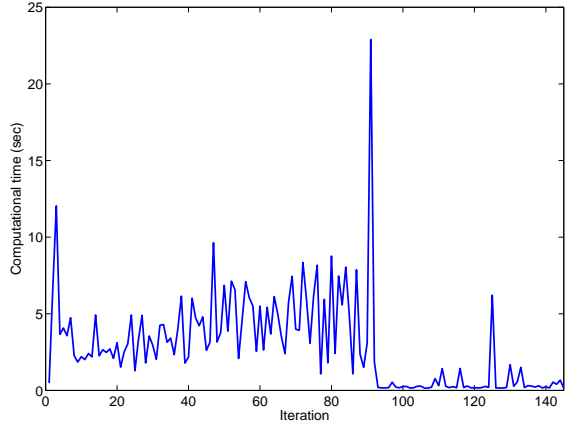}
\end{center}
\caption{Computational time in seconds for each discrete time step.}\label{tempo}
\end{figure}

\section{Appendix}

\subsection{Proof of Theorem \ref{existthm}}
This proof uses a similar approach as for \cite[Theorem 2.3]{fowa10}. Let us first consider a partition $\mathscr P =\{\mathcal U_{\mathcal I_0^j}\}_{j=1}^{2^m}$
of $\EE$ indexed by all subsets $ \mathcal I_0^j \subset \mathcal I$, as follows
$$
\mathcal U_{\mathcal I_0^j} = \{ v \in \EE: |v_i| \leq r+\varepsilon, i \in \mathcal I_0^j, |v_i| >r +\varepsilon, i \in \mathcal I \setminus \mathcal I_0^j \}.
$$
The minimization of $\J_p^\varepsilon$ over $\mathcal F(f) \cap \mathcal U_{\mathcal I_0^j}$ can be reformulated as 
\begin{equation}\label{minfrankwolfe}
\left \{
\begin{array}{l} 
\mbox{Minimize }  \bar \J_p^\varepsilon(v) = \|T v - g \|^2 + \gamma \sum_{i=1}^m c_i \bar W_r^{p,\varepsilon} (v_i), \mbox{ Subject to } v \in \mathcal F(f) \cap \mathcal U_{\mathcal I_0^j},\\
c_i=0 \mbox{ if } i \in  \mathcal I \setminus \mathcal I_0^j \mbox{ and } c_i =1 \mbox { if } i \in \mathcal I_0^j,
\end{array}
\right .
\end{equation}
where
$$
\bar W_r^{p,\varepsilon}(t) = 
\left \{
\begin{array}{ll}
t^p, & t \leq r-\varepsilon,\\
\pi_p(t), & r-\varepsilon \leq t \leq r+\varepsilon,\\
|t-\varepsilon|^p & t \geq r+\varepsilon.
\end{array}
\right .
$$
If we prove that the minimization \eqref{minfrankwolfe} has always a solution $v(\mathcal I_0^j)$ for all $j=1,\dots,2^m$, and such a minimizer belongs to a compact set $M^j$, independent
of $\varepsilon \geq 0$, then
$$
v^* = \arg \min_{j=1,\dots,2^m} \J_p^\varepsilon(v(\mathcal I_0^j)),
$$ 
is actually a solution for \eqref{minprob3} and it belongs to the compact set $M = \cup_{j=1}^{2^m} M^j$, independent
of $\varepsilon \geq 0$. Hence, it is sufficient now to address \eqref{minfrankwolfe}.
For that, we first show the following technical observation:\\

 If $x,v \in \EE$ are fixed and $\bar \J_p^\varepsilon$ is bounded above and below
on the ray $R_{x,v}=\{x + t v, t \geq 0\}$, then $\bar \J_p^\varepsilon$ is actually constant on $R_{x,v}$.
In fact, let us consider the function $\mu(t) = \bar \J_p^\varepsilon(x + t v)$. 
By the boundedness of $\bar \J_p^\varepsilon(x + t v)$, without loss of generality, we can assume that $0 \leq \mu(t) \leq 1$. 
Hence there exists a sequence $(t_n)_n \subset \mathbb R^+$
of points $t_n \to + \infty$ for $n \to \infty$ such that $\mu(t_n) \to \eta \in [0,1]$ for $n \to \infty$. 
Moreover, by definition of $\bar W_r^{p,\varepsilon}$, for $t>0$ sufficiently large we have actually the general expression $\mu(t) = P(t) + \gamma \sum_{i=1}^m c_i | x_i - \varepsilon + t v_i|^p$,
where $P$ is a polynomial of degree at most $2$. Assume now, for instance, that $1 \leq p \leq 2$. As $\lim_n \frac{\mu(t_n)}{t_n^2} =0$ we deduce that all the coefficients in $P$ 
 of second degree are actually vanishing. In turn, then $0 = \lim_n \frac{\mu(t_n)}{|t_n|^p}$ has the implication that for each $i$ one of the coefficients $c_i$ or $d_i$
must vanish as well. Following in the same manner, we conclude that all linear coefficients in $\mu(t)$ also vanish, leaving only the possibility that $\mu(t)$ is a constant function.
A similar approach can be conducted to prove the observation also for $p>2$. \\

Notice now that  $\bar \J_p^\varepsilon$ converges uniformly to $\bar \J_p$ on $\mathcal U_{\mathcal I_0^j}$ for $\varepsilon \to 0$, as defined in \eqref{minprob2}, or
\begin{equation}\label{unifconv}
| \bar \J_p^\varepsilon(v) - \bar \J_p(v) | \leq \Gamma(\varepsilon), \quad \mbox{for all } v \in \mathcal U_{\mathcal I_0^j},
\end{equation}
for a continuous function $\Gamma(\varepsilon)=o(\varepsilon)$, $\varepsilon \to 0$. By Remark \ref{exrem}, for $X=\mathcal F(f) \cap \overline{\mathcal U_{\mathcal I_0^j}}$ and any $v^0 \in X$, there exists a linear subspace $\mathcal V \subset  \EE$, such that the orthogonal projection $X^\perp$ of $X$ onto $\mathcal V^\perp$ has the properties 
\begin{itemize}
\item $X = \{ x =  x^\perp \oplus t v: x^\perp \in X^\perp, v \in \mathcal V, t \in \mathbb R^+\}$,
\item $M_{C}^j = X^\perp  \cap \{ v \in \EE: \bar \J_p(v) \leq C \}$, for $C \geq \bar \J_p^\varepsilon(v^0)+\Gamma(\varepsilon)$ is compact, and 
\item $\bar \J_p (\xi_t)$ is constant along rays $\xi_t=x^\perp \oplus t v $, where $x^\perp \in M_{C}^j$, $v \in  \mathcal V$, and $ t \in \mathbb R^+$. 
\end{itemize}
By the uniform estimate \eqref{unifconv} and the last property, we deduce that $\bar \J_p^\varepsilon(\xi_t)$ is bounded from above and below by $\bar \J_p (x^\perp)\pm \Gamma(\varepsilon)$ on rays $\xi_t=x^\perp \oplus t v $, where $x^\perp \in M_{C}$, $v \in  \mathcal V$, and $ t \in \mathbb R^+$. Hence, we conclude that $\bar \J_p^\varepsilon(\xi_t)$ is also constant for $t\geq 0$. From \eqref{unifconv}, the set
$$
  X^\perp  \cap \{ v \in \EE: \bar \J_p^\varepsilon (v) \leq \bar \J_p^\varepsilon(v^0)\},
$$
is included in $M_{C}^j$, and 
$$
\inf_{v \in \mathcal F(f) \cap \mathcal U_{\mathcal I_0^j}} \bar \J_p^\varepsilon(v) =\inf_{v \in M^j_C} \bar \J_p^\varepsilon(v).
$$
By compactness of $M^j=M^j_C$ and continuity of $\bar \J_p^\varepsilon$ we conclude the existence of minimizers in $M^j$. As pointed out above, this further implies the existence of minimal solutions
in $M = \cup_{j=1}^{2^m} M^j$ of the original problem \eqref{minprob3}. Notice further that, by continuity of $ \bar \J_p^\varepsilon(v^0)+\Gamma(\varepsilon)$ with respect to $\varepsilon$, the sets $M^j=M^j_C$ actually do not depend on $0\leq \varepsilon$ as soon as $C\geq \max_{0 < \varepsilon} \bar \J_p^\varepsilon(v^0)+\Gamma(\varepsilon)$ is large enough.

\bibliography{mf}
\bibliographystyle{plain}

\end{document}